\documentclass[reqno]{amsart}
\usepackage{amssymb}
\usepackage{hyperref}
\usepackage{booktabs}
 \usepackage{multirow}

 \usepackage{color}

\newtheorem{theorem}{Theorem}[section]
\newtheorem{proposition}[theorem]{Proposition}
\newtheorem{lemma}[theorem]{Lemma}
\newtheorem{corollary}[theorem]{Corollary}

\theoremstyle{remark}
\newtheorem{remark}{Remark}[section]

\newtheorem*{note}{Note}

\numberwithin{equation}{section}

\begin{document}

\title[Cubature rules from Hall-Littlewood polynomials]
{Cubature rules from Hall-Littlewood polynomials}

\author{J.F.  van Diejen}

\address{
Instituto de Matem\'atica y F\'{\i}sica, Universidad de Talca,
Casilla 747, Talca, Chile}

\email{diejen@inst-mat.utalca.cl}

\author{E. Emsiz}

\address{
Delft Institute of Applied Mathematics,
Delft University of Technology,
Van Mourik Broekmanweg 6, 2628 XE, Delft, The Netherlands}
\email{e.emsiz@tudelft.nl}


\thanks{This work was supported in part by the {\em Fondo Nacional de Desarrollo
Cient\'{\i}fico y Tecnol\'ogico (FONDECYT)} Grant   \# 1170179.}

\date{November 2019}

\begin{abstract}
Discrete orthogonality relations for  Hall-Littlewood polynomials are employed, so as to derive cubature rules 
for the integration of homogeneous symmetric functions with respect to the density of the circular unitary ensemble (which originates from the
Haar measure on the special unitary group $SU(n;\mathbb{C})$). By passing to Macdonald's hyperoctahedral Hall-Littlewood polynomials,
we moreover find analogous cubature rules for the integration with respect to the density of the circular quaternion ensemble (which originates in turn from the
Haar measure on the compact symplectic group $Sp (n;\mathbb{H})$).
The cubature formulas under consideration are exact for a class of rational symmetric functions with simple poles supported on a prescribed complex hyperplane arrangement.
 In the planar situations (corresponding to $SU(3;\mathbb{C})$ and $Sp (2;\mathbb{H})$),  a determinantal expression for the Christoffel weights enables us to write down compact
cubature rules for the integration over the equilateral triangle and the isosceles  right triangle, respectively.
\end{abstract}

\maketitle

%\tableofcontents

%%%%%%%%%%%%%%%%%%%%%%%%%%%%%%%%%%%%%%%%%%%%
%%%%%%%%%%%%%%% SECTION %%%%%%%%%%%%%%%%%%%%%%%
%%%%%%%%%%%%%%% SECTION %%%%%%%%%%%%%%%%%%%%%%%
%%%%%%%%%%%%%%% SECTION %%%%%%%%%%%%%%%%%%%%%%%
%%%%%%%%%%%%%%%%%%%%%%%%%%%%%%%%%%%%%%%%%%%%

\section{Introduction}\label{sec1}
It is well-known that the Haar measures of the classical compact Lie groups  \cite{sim:representations,pro:lie} yield the densities of ubiquitous random matrix ensembles \cite{meh:random,for:log-gases}.  
A crucial issue, from the point of view of applications, is the development of techniques that permit to perform efficient numerical integration with respect to the densities in question. 
In recent years, Gauss-like cubature rules were constructed serving this purpose \cite{mun:group,li-xu:discrete,moo-pat:cubature,moo-mot-pat:gaussian,hri-mot:discrete,hri-mot-pat:cubature},
with the aid of a fundamental  toolset based on the use of orthogonal polynomials
\cite{str:approximate,hof-wit:generalized,bee:chebyshev,sob:cubature,sob-vas:theory,coo:constructing,coo-mys-sch:cubature,dun-xu:orthogonal}. For the pertinent class of integrals at issue, the cubature nodes arise in this picture from the
zeros of characters of irreducible representations of the underlying Lie group. These characters are given explicitly by Schur polynomials, 
and the aim of the present work is to generalize the corresponding construction from Schur polynomials to Hall-Littlewood
polynomials \cite{mac:symmetric,mac:orthogonal,nel-ram:kostka}. To this end we exploit discrete orthogonality structures for the Hall-Littlewood polynomials originating from mathematical physics \cite{die:diagonalization,die:finite-dimensional,die-ems:orthogonality}.
Our approach entails cubature rules for the integration  of rational symmetric functions with prescribed poles on a complex hyperplane arrangement, controlled
by the orthogonality measure of the Hall-Littlewood polynomials. In the special case of a rank-one Lie group, we reproduce in this manner particular instances of known quadrature rules stemming from the Bernstein-Szeg\"o polynomials \cite{dar-gon-jim:quadrature,bul-cru-dec-gon:rational,die-ems:quadrature}, which were
conceived to integrate rational functions with prescribed poles against the Chebyshev weight functions.
The material is organized as follows.

In Section \ref{sec2} we formulate our cubature rule stemming from the Hall-Littlewood polynomials. The formula in question serves to
integrate homogeneous symmetric functions with respect to the density of the circular unitary ensemble, given
by the Haar measure on the special unitary group $SU(n;\mathbb{C})$.

In Section \ref{sec3} we provide an analogous construction  based on Macdonald's hyperoctahedral Hall-Littlewood polynomials.
The corresponding cubature rule is designed to integrate symmetric functions with respect to the density of the circular quaternion ensemble, which is given in turn by the
Haar measure on the compact symplectic group $Sp (n;\mathbb{H})$.

In both situations the cubature nodes turn out to be located at common roots of an associated family of quasi-orthogonal polynomials. An explicit formula for the quasi-orthogonal 
polynomials of interest is derived in Section \ref{sec4}.

The Christoffel weights of our cubature rules are encoded by squared norms determined by discrete orthogonality relations for the (hyperoctahedral) Hall-Littlewood polynomials
from Refs. \cite{die:diagonalization,die:finite-dimensional,die-ems:orthogonality}.
In Section \ref{sec5}, we formulate a compact determinantal formula for these Christoffel weights
in the case of planar integrals (associated with  $SU(3;\mathbb{C})$ and $Sp (2;\mathbb{H})$). The corresponding
cubature rules serve to integrate over the equilateral triangle and the isosceles  right triangle, respectively. 

Section \ref{sec6} concludes our presentation, by pointing out how various previous cubature rules studied in 
\cite{mun:group,li-sun-xu:discrete,li-xu:discrete,moo-pat:cubature,moo-mot-pat:gaussian,hri-mot:discrete,hri-mot-pat:cubature,die-ems:exact,die-ems:cubature} can be seen as parameter degenerations of those considered here.  The Hall-Littlewood polynomials specialize at the parameter values of interest to (symplectic) Schur polynomials or to symmetric monomials, respectively.

\begin{note}
Below we will occasionally refer to the dominance partial ordering of vectors  in $\mathbb{R}^n$:
\begin{equation}
\mathbf{x}\leq \mathbf{y}   \Longleftrightarrow  x_1+\cdots +x_k\leq y_1+\cdots +y_k\quad (k=1,\ldots,n).
\end{equation}
We will also employ
the following notation for counting the multiplicity of $x\in\mathbb{R}$ inside $\mathbf{x}=(x_1,\ldots ,x_n)\in\mathbb{R}^n$:
\begin{equation}\label{mult}
\text{m}_x(\mathbf{x}):= | \{  1\leq j\leq n\mid x_j=x\} |  .
\end{equation}

\end{note}

\section{Cubature rules associated with Hall-Littlewood polynomials}\label{sec2}
In this section we present a cubature rule for the evaluation of integrals of homogeneous symmetric functions in the variables $z_j=e^{i\xi_j}$ ($j=1,\ldots ,n$), over the fundamental domain
\begin{equation}\label{A:a}
\mathbb{A}^{(n)}_{\texttt{a}}:= \{ (\xi_1,\ldots ,\xi_n)\in \mathbb{R}^{n}_0 \mid  \xi_1>\xi_2>\cdots >\xi_{n}> \xi_1- 2\pi\}
\end{equation}
for the hyperplane
\begin{equation}
\mathbb{R}^{n}_0:= \{ (\xi_1,\ldots\xi_n) \in\mathbb{R}^{n}\mid \xi_1+\cdots +\xi_{n}=0\}. 
\end{equation}
Here the integration is with respect to the density of the circular 
unitary ensemble 
\begin{equation}\label{circular}
\rho_{\texttt{a}}(\boldsymbol{\xi}):=   \prod_{1\leq j<k\leq n}    |   e^{i\xi_j}  -e^{i\xi_k}  |^2  =  2^{n(n-1)} \prod_{1\leq j<k\leq n} \sin^2\left(  \frac{\xi_j-\xi_k}{2}\right) 
\end{equation}
stemming from the
Haar measure on the special unitary group $SU(n;\mathbb{C})$. The coordinates $\boldsymbol{\xi}:=(\xi_1,\ldots ,\xi_n)$ correspond in this picture to the angles of the eigenvalues.

\subsection{Hall-Littlewood polynomials} The Hall-Littlewood polynomials constitute an important orthogonal basis for the space of symmetric functions in $n$ variables, which has been studied intensively from the point of algebraic combinatorics through its connections with the representation theory of affine Hecke algebras. For our purposes it suffices to collect only a few elementary properties extracted from the standard references \cite[Chapter III]{mac:symmetric}, \cite[\S 10]{mac:orthogonal} and \cite{nel-ram:kostka}.

A convenient way to label 
Hall-Littlewood polynomials is by means of dominant weight vectors
\begin{subequations}
\begin{equation}\label{dominant-cone:a}
\Lambda^{(n)}_{\texttt{a}}:= \{   l_1\omega_1+\cdots+l_{n-1}\omega_{n-1} \mid l_1,\ldots ,l_{n-1}\in\mathbb{Z}_{\geq 0} \} 
\end{equation}
that are generated by the $SU(n;\mathbb{C})$ fundamental weight basis  (cf. \cite[Planche I]{bou:groupes}) 
\begin{equation}\label{fwb:a}
\omega_j=e_1+\cdots +e_j- {\textstyle \frac{j}{n}}(e_1+\cdots +e_{n})\qquad (j=1,\ldots ,n-1)
\end{equation}
\end{subequations}
spanning the hyperplane $\mathbb{R}^{n}_0$. (Here the vectors $e_1,\ldots ,e_n$ refer to the standard unit basis of $\mathbb{R}^n$). Specifically, for any $\mu=(\mu_1,\ldots,\mu_n)\in \Lambda^{(n)}_{\texttt{a}}$ the corresponding Hall-Littlewood polynomial is given explicitly by
\begin{subequations}
\begin{align}\label{HLp:a}
P_{\texttt{a};\mu} (\boldsymbol{\xi} ;q) :=  
 \sum_ {\sigma\in S_{n}}   C_{\texttt{a}}( \xi_{\sigma_1},\ldots ,  \xi_{\sigma_{n}};q)
\exp (i \xi_{\sigma_1}\mu_1+\cdots +i  \xi_{\sigma_{n}} \mu_{n})  ,
\end{align}
where
\begin{equation}\label{Cp:a}
C_{\texttt{a}}(\xi_1,\ldots ,\xi_n;q)=C_{\texttt{a}}(\boldsymbol{\xi};q) := \prod_{1\leq j<k \leq n} \left(\frac{1-q e^{-i(\xi_{j}-\xi_k)}}{1-e^{-i(\xi_{j}-\xi_k)}}\right) ,
\end{equation}
\end{subequations}
and the summation is meant over all permutations $\sigma= { \bigl( \begin{smallmatrix}1& 2& \cdots & n \\
 \sigma_1&\sigma_2&\cdots &\sigma_n
 \end{smallmatrix}\bigr)}$ comprising the symmetric group $S_{n}$. 
For $-1<q<1$, Hall-Littlewood polynomials are known to obey the following fundamental orthogonality relations, cf. e.g. \cite[\S 10]{mac:orthogonal} or \cite[Section 3]{nel-ram:kostka} (with the root system $R$ of type $A_{n-1}$):
\begin{align}\label{or-c:a}
\frac{1}{(2\pi )^{n-1} n^{1/2}}   \int_{\mathbb{A}^{(n)}_{\texttt{a}}}    P_{\texttt{a};\mu}  (\boldsymbol{\xi};q) \overline{P_{\texttt{a};\nu}  (\boldsymbol{\xi};q )}     | C_{\texttt{a}}(\boldsymbol{\xi};q) |^{-2}  \text{d} 
\boldsymbol{\xi} &  \\
=
\begin{cases} 
 \prod_{\substack{1\leq j<k\leq n\\ \mu_j-\mu_k=0}}  \frac{1-q^{1+k-j}}{1-q^{k-j}} &\text{if}\  \nu= \mu ,\\
0 &\text{if}\  \nu\neq \mu 
\end{cases}
 & \nonumber
\end{align}
($\mu,\nu\in \Lambda^{(n)}_{\texttt{a}}$).

\begin{remark}
In the orthogonality relations \eqref{or-c:a}  the integration is meant with respect to the Lebesgue  measure $\text{d}\boldsymbol{\xi}$ stemming from the standard volume form associated with the $(n-1)$-dimensional euclidean space $\mathbb{R}^n_0$. In particular:
$\int_{\mathbb{A}^{(n)}_{\texttt{a}}} \text{d}\boldsymbol{\xi} = \text{Vol} \bigl(\mathbb{A}^{(n)}_{\texttt{a}}\bigr) = \frac{(2\pi)^{n-1}n^{1/2}}{n!} $.
\end{remark}

\subsection{Finite-dimensional orthogonality relations}
Given a fixed positive integral level $m$, we consider the following finite alcove in $\Lambda^{(n)}_{\texttt{a}}$:
\begin{equation}\label{dwv:a}
\Lambda^{(m,n)}_{\texttt{a}}:= \{   l_1\omega_1+\cdots+l_{n-1}\omega_{n-1} \mid l_1,\ldots ,l_{n-1}\in\mathbb{Z}_{\geq 0},\, l_1+\cdots +l_{n-1}\leq m\} .
\end{equation}
In \cite{die:diagonalization}
a lattice Laplacian on $\Lambda^{(m,n)}_{\texttt{a}}$ was constructed  (with Robin-type boundary conditions involving the parameter $q$)  for which
$P_{\texttt{a};\mu} (\boldsymbol{\xi}) $ (viewed as a function of $\mu\in \Lambda^{(m,n)}_{\texttt{a}}$) constitutes an eigenfunction provided the spectral variable $\boldsymbol{\xi}\in \mathbb{A}^{(n)}_{\texttt{a}}$ belongs
to a discrete set of nodes $ \boldsymbol{\xi}^{(m,n)}_{\texttt{a};\lambda}$, $\lambda\in \Lambda^{(m,n)}_{\texttt{a}}$ parametrizing the eigenvalues. For $-1<q<1$, the construction in question gave rise to a novel finite-dimensional orthogonality relation for the Hall-Littlewood polynomials of the form  \cite[Section 5.2]{die:diagonalization}:
\begin{subequations}
\begin{equation}\label{do:a}
\sum_{\mu\in\Lambda^{(m,n)}_{\texttt{a}}}      P_{\texttt{a};\mu} \bigl( \boldsymbol{\xi}^{(m,n)}_{\texttt{a};\lambda} ;q \bigr)     \overline{  P_{\texttt{a};\mu} \bigl( \boldsymbol{\xi}^{(m,n)}_{\texttt{a};\kappa}  ;q \bigr) }
 \delta^{(m,n)}_{\texttt{a};\mu }(q) = 0 \quad\text{if}\ \lambda\neq\kappa
\end{equation}
($\lambda,\kappa\in \Lambda^{(m,n)}_{\texttt{a}}$), where
 \begin{equation}\label{w:a}
\delta^{(m,n)}_{\texttt{a};\mu }(q):= \prod_{\substack{1\leq j<k\leq n\\ \mu_j-\mu_k=0}}  \frac{1-q^{k-j}}{1-q^{1+k-j}}
\prod_{\substack{1\leq j<k\leq n\\ \mu_j-\mu_k=m}}  \frac{1-q^{n-k+j}}{1-q^{n+1-k+j}}  .
\end{equation}
\end{subequations}

\subsection{Positions of the nodes} At general parameter values $-1<q<1$, explicit formulas for the positions of the spectral nodes $\boldsymbol{\xi}^{(m,n)}_{\texttt{a};\lambda}$, $\lambda\in \Lambda^{(m,n)}_{\texttt{a}}$ are not available unfortunately. Instead, we will recur to a simple numerical algorithm stemming from  \cite[Section 4]{die:diagonalization}.
Specifically, for any $-1<q<1$ and $\lambda=(\lambda_1,\ldots,\lambda_n)\in \Lambda^{(m,n)}_{\texttt{a}}$ the explicit position of the pertinent
node $\boldsymbol{\xi}^{(m,n)}_{\texttt{a};\lambda}$ turns out to be given by  the
unique global minimum of the following  semi-bounded Morse function $V^{(m,n)}_{ \texttt{a};\lambda }:\mathbb{R}^n\to \mathbb{R}$:
\begin{subequations}
\begin{equation}\label{morse-a}
V^{(m,n)}_{ \texttt{a};\lambda }(\boldsymbol{\xi}) := \sum_{1\le j < k \le n }    \int_0^{\xi_j-\xi_k} v_q(\vartheta)\text{d}\vartheta 
 + \sum_{1\leq j\leq n} \left(
{\textstyle \frac{m}{2} }\xi_j^2-2\pi (\varrho_{\texttt{a};j}+\lambda_j)\xi_j 
 \right) ,
\end{equation}
where
\begin{equation}
\varrho_{\texttt{a};j}:= \frac{1}{2} \bigl( n+1-2j\bigr)\qquad  (j=1,\ldots ,n) 
\end{equation}
and
\begin{equation}\label{uv}
v_q(\vartheta ) := 
\int_0^\vartheta u_q (\theta) \text{d}\theta\quad \text{with}\quad
u_q(\theta) := \frac{1-q^2}{1-2q\cos (\theta) +q^2}.
\end{equation}
\end{subequations}
Notice in this connection that the existence of this global minimum  is guaranteed because $V^{(m,n)}_{\texttt{a},\lambda}  (\boldsymbol{\xi})\to +\infty$ as $|\boldsymbol{\xi}|\to\infty$, whereas the uniqueness follows by convexity:
\begin{align}\label{Hesse:a}
&H^{(n,m)}_{\texttt{a};j,k}(\boldsymbol{\xi)}:=\partial_{\xi_j}\partial_{\xi_k} V^{(n,m)}_{ \texttt{a};\lambda} (\boldsymbol{\xi}) \\
&=
\begin{cases}
m + \sum_{\substack{1\leq l\leq n\\ l\neq j}}u_q(\xi_j-\xi_l) & \text{if $ k=j$}\\
-u_q(\xi_j-\xi_k) & \text{if $k\neq j$}\\
\end{cases} ,\nonumber
\end{align}
so (for any $(x_1,\ldots ,x_n)\in\mathbb{R}^n$)
\begin{align*}
\sum_{1\leq j,k\leq n}  H^{(m,n)}_{\texttt{a}; j,k} (\boldsymbol{\xi}) x_j x_k   
= & \sum_{1\leq j\leq n}m x_j^2 +  \sum_{1\leq j<k\leq n}
   u_q(\xi_j -\xi_k)(x_j-x_k)^2 \\ \ge &\,  m\sum_{1\leq j\leq n} x_j^2 .
\end{align*}

The numerical positions of the nodes can now be conveniently computed from the equations for the critical point $\partial_{\xi_j} V^{(m,n)}_{ \texttt{a};\lambda }(\boldsymbol{\xi}) =0$:
\begin{equation}\label{CEQ:a}
m\xi_j+   \sum_{\substack{ 1\leq k\leq n \\ k\neq j}} v_q(\xi_j-\xi_k)=2\pi (\lambda_j+\varrho_{\texttt{a},j})\qquad (j=1,\ldots ,n),
\end{equation}
by means of a fixed-point iteration scheme such as Newton's method. At this point numerical integration for the evaluation of $v_q(\xi)$ is to be avoided, since it is much more efficient to invoke the explicit formula
$v_q(\vartheta ) =2\arctan \left(  \frac{1+q}{1-q} \tan \bigl(\frac{\vartheta}{2}\bigr)      \right)$ for $-\pi<\vartheta<\pi$, in combination with
the quasi-periodicity  $v_q(\vartheta+2\pi)= v_q(\vartheta)+2\pi$ for $\vartheta\in\mathbb{R}$.
At the special parameter value $q=0$ Eq. \eqref{CEQ:a} degenerates into a linear system,  the solution of which is given explicitly by $\xi_j=\frac{2\pi  (\lambda_j+\varrho_{\texttt{a},j}) }{n+m}$, $j=1,\ldots ,n$;
 this explicit solution at  $q=0$ serves as a suitable initial estimate for starting up the Newton iteration at general parameter values $-1 < q <1$ (cf. Remark \ref{bounds-a:rem} below).

\begin{remark}\label{regularity-a:rem}
It is instructive to observe that
the nodes $\boldsymbol{\xi}^{(m,n)}_{\texttt{a};\lambda}$, $\lambda\in\Lambda^{(m,n)}_{\texttt{a}}$ belong to the domain $\mathbb{A}^{(n)}_{\texttt{a}}$ \eqref{A:a}. 
Indeed, by summing the $n$ equations  in Eq. \eqref{CEQ:a} characterizing  the position of $\boldsymbol{\xi}^{(m,n)}_{\texttt{a};\lambda}$, one sees---upon exploiting that $v_q(\vartheta)$ is odd in $\vartheta$---that the critical point in question lies on the hyperplane $\mathbb{R}^n_0$. Furthermore, by subtracting the $k$th equation from the $j$th  equation:
\begin{equation}\label{gap-a:a}
m(\xi_j-\xi_k) +  \sum_{1\leq l\leq n}  \Bigl( v_q(\xi_j-\xi_l) -v_q(\xi_k-\xi_l) \Bigr) =  2\pi (\lambda_j  -\lambda_k + k-j),
\end{equation}
it is manifest that at $\boldsymbol{\xi}=\boldsymbol{\xi}^{(m,n)}_{\texttt{a};\lambda}$ the inequality $2\pi>\xi_j-\xi_k>0$ holds when $1\leq j<k\leq n$.
Here one uses the monotonicity and the (above) quasi-periodicity of $v_q(\vartheta)$ in $\vartheta$, together with the observation that in this situation $0\leq \lambda_j-\lambda_k\leq m$ (because $\lambda\in\Lambda^{(m,n)}_{\texttt{a}}$).
\end{remark}

\begin{remark}\label{bounds-a:rem}
Since $ \frac{1-|q|}{1+|q|} \leq u_q(\theta) \leq \frac{1+|q|}{1-|q|}$ for $\theta\in\mathbb{R}$,
the following bounds for the position of the node $\boldsymbol{\xi}=\boldsymbol{\xi}^{(m,n)}_{\texttt{a};\lambda}$ ($\lambda\in\Lambda^{(m,n)}_{\texttt{a}}$) are immediate from
Eq. \eqref{gap-a:a} via the mean value theorem:
\begin{subequations}
  \begin{equation}\label{gap-a:b}
\frac{2\pi(k-j+\lambda_j-\lambda_k)}{m+\kappa_{\texttt{a};-}(q)} \leq \xi_j-\xi_k\leq \frac{2\pi(k-j+\lambda_j-\lambda_k)}{m+\kappa_{\texttt{a};+}(q)} 
\end{equation}
for $1 \le  j < k \le n$, where
\begin{equation}\label{gap-a:c}
\kappa_{\texttt{a};\pm}(q):= n \left( \frac{1-|q|}{1+|q|}\right)^{\pm 1} .
\end{equation}
\end{subequations}
These bounds confirm that
\begin{equation}
\left.  \boldsymbol{\xi}^{(m,n)}_{\texttt{a};\lambda}\right|_{q=0}=\frac{2\pi (\lambda+\varrho_{\texttt{a}})}{n+m}\qquad (\lambda\in\Lambda^{(m,n)}_{\texttt{a}}),
\end{equation}
where $\varrho_{\texttt{a}}:=(\varrho_{\texttt{a},1},\ldots , \varrho_{\texttt{a},n})$.
Moreover, since at $\boldsymbol{\xi}=\frac{2\pi (\lambda+\varrho_{\texttt{a}})}{n+m}$ ($\lambda\in\Lambda^{(m,n)}_{\texttt{a}}$) the inequalities in Eqs. \eqref{gap-a:b}, \eqref{gap-a:c} are satisfied for any
$-1<q<1$, this special point provides a convenient initial estimate when computing the position of the node $\boldsymbol{\xi}^{(m,n)}_{\texttt{a};\lambda}$ numerically from Eq. \eqref{CEQ:a} via Newton's method.
\end{remark}

\subsection{Cubature rule}\label{cub:a:sec}
Let
\begin{equation}
\mathbb{P}^{(m,n)}_{\texttt{a}}:=  \text{Span}_{\mu\in\Lambda^{(m,n)}_{\texttt{a}}}   \{    M_{\texttt{a};\mu} (\boldsymbol{\xi})  \}  ,
\end{equation}
with
\begin{subequations}
\begin{equation}\label{smonomials}
M_{\texttt{a};\mu} (\boldsymbol{\xi}) :=  \frac{1}{N_{\texttt{a}; \mu}}   \sum_ {\sigma\in S_{n}}   
\exp (i \xi_{\sigma_1}\mu_1+\cdots +i  \xi_{\sigma_{n}} \mu_{n})  
\end{equation}
normalized such that each exponential term on the RHS occurs with multiplicity one:
\begin{equation}\label{stabilizer}
  N_{\texttt{a};\mu} :=  \prod_{\substack{1\leq j<k\leq n\\ \mu_j-\mu_k=0}}  \frac{1+k-j}{k-j} .
  \end{equation}
\end{subequations}
Notice that the space $\mathbb{P}^{(m,n)}_{\texttt{a}}$ is isomorphic to the $\binom{m+n-1}{m}$-dimensional space of symmetric polynomials of degree at most $m$ in each
of the variables $z_j=e^{i\xi_j}$ ($j\in \{ 1,\ldots ,n\}$) subject to the relation $z_1\cdots z_n=1$.

\begin{theorem}[Hall-Littlewood Cubature]\label{hwc-a:thm}
For $q\in (-1,1)$ and $m\in\mathbb{Z}_{>0}$,  the following cubature rule holds true for any  symmetric polynomial
$f(\boldsymbol{\xi})$  in $\mathbb{P}^{(2m-1,n)}_{\texttt{a}}$:
\begin{subequations}
\begin{equation}\label{hwc:a}
\frac{1}{(2\pi )^{n-1} n^{1/2}}   \int_{\mathbb{A}^{(n)}_{\texttt{a}}}   f(\boldsymbol{\xi})  | C_{\texttt{a}}(\boldsymbol{\xi};q) |^{-2} 
  \text{d} \boldsymbol{\xi}
=
\sum_{\lambda\in\Lambda^{(m,n)}_{\texttt{a}}}      f \bigl( \boldsymbol{\xi}^{(m,n)}_{\texttt{a};\lambda}  \bigr)   \hat{\Delta}^{(m,n)}_{\texttt{a};\lambda } ,
\end{equation}
with Christoffel weights given by
\begin{equation}\label{cw:a}
\hat{\Delta}^{(m,n)}_{\texttt{a};\lambda }:= 
\Biggl(\sum_{\mu\in \Lambda^{(m,n)}_{\texttt{a}}}   \left| P_{\texttt{a};\mu} \bigl(\boldsymbol{\xi}^{(m,n)}_{\texttt{a};\lambda} ;q\bigr) \right|^2  \delta^{(m,n)}_{\texttt{a};\mu } (q) \Biggr)^{-1}  .
\end{equation}
\end{subequations}
\end{theorem}

\begin{proof}
It is immediate from the discrete orthogonality relations in Eqs. \eqref{do:a}, \eqref{w:a}  that the  following matrix is unitary:
\begin{equation*}
\left[ 
\sqrt{ \delta^{(m,n)}_{\texttt{b};\mu }(q)} P_{\texttt{a};\mu} \bigl( \boldsymbol{\xi}^{(m,n)}_{\texttt{a};\lambda} ;q\bigr)  \sqrt{\hat{\Delta}^{(m,n)}_{\texttt{a};\lambda }}
\right]_{\mu,\lambda\in \Lambda^{(m,n)}_{\texttt{a}}}  .
\end{equation*}
By `column-row duality' this means that for any $\mu,\nu\in \Lambda^{(m,n)}_{\texttt{a}}$  (cf. \cite[Theorem 1]{die:finite-dimensional}):
\begin{equation}\label{or-d:a}
\sum_{\lambda\in\Lambda^{(m,n)}_{\texttt{a}}}      P_{\texttt{a};\mu} \bigl( \boldsymbol{\xi}^{(m,n)}_{\texttt{a};\lambda} ;q \bigr)     \overline{  P_{\texttt{a};\nu} \bigl( \boldsymbol{\xi}^{(m,n)}_{\texttt{a};\lambda}  ;q \bigr) }
 \hat{\Delta}^{(m,n)}_{\texttt{a};\lambda }  
=
\begin{cases} 
1/ \delta^{(m,n)}_{\texttt{a};\mu }(q)&\text{if}\  \nu= \mu , \\
0 &\text{if}\  \nu\neq \mu .
\end{cases} 
\end{equation}
If we compare this formula with the standard orthogonality relations for the corresponding Hall-Littlewood polynomials in Eq. \eqref{or-c:a},
then it is clear that both scalar products are equal if $\nu$ (say) is restricted to $ \Lambda^{(m-1,n)}_{\texttt{a}}$. Hence, since
$  \overline{P_{\texttt{a};\nu}  (\boldsymbol{\xi};q )}  =  P_{\texttt{a};(-\nu_n,\ldots ,-\nu_1)}  (\boldsymbol{\xi};q )$ and $(-\nu_n,\ldots ,-\nu_1)=l_{n-1}\omega_1+l_{n-2}\omega_2+\dots +l_1\omega_{n-1}$ when
$\nu= l_{1}\omega_1+l_{2}\omega_2+\dots +l_{n-1}\omega_{n-1}$),
we conclude that the asserted cubature rule is valid for all symmetric polynomials $f(\boldsymbol{\xi})$ of the form
\begin{equation}\label{kostka}
f(\boldsymbol{\xi})=     P_{\texttt{a};\mu}  (\boldsymbol{\xi};q) P_{\texttt{a};\nu}  (\boldsymbol{\xi};q )   \quad\text{with}\quad \mu\in  \Lambda^{(m,n)}_{\texttt{a}},\ \nu\in  \Lambda^{(m-1,n)}_{\texttt{a}}.
\end{equation}
The products in question actually span $\mathbb{P}^{(2m-1,n)}_{\texttt{a}}$
(because the monomial expansion of $f(\boldsymbol{\xi})$ \eqref{kostka} contains $M_{\texttt{a};\mu+\nu}(\boldsymbol{\xi})$ and monomial symmetric functions $M_{\texttt{a};\kappa}(\boldsymbol{\xi})$ corresponding to dominant weights $\kappa$ that are smaller than $\mu+\nu$ in the
dominance partial order).
The cubature rule thus  follows for general symmetric polynomials $f(\boldsymbol{\xi})$ in $\mathbb{P}^{(2m-1,n)}_{\texttt{a}}$ by linearity.
\end{proof}

The following corollary interprets Theorem \ref{hwc-a:thm} as an exact cubature rule for the integration
of a class of rational symmetric  functions against density of the circular  unitary ensemble (CUE).

\begin{corollary}[Cubature in CUE]\label{hwc-a:cor}
For $q\in (-1,1)$ and $m\in\mathbb{Z}_{>0}$, one has that
\begin{subequations}
\begin{equation*}\label{hwcr:a}
\frac{1}{(2\pi )^{n-1} n^{1/2}}   \int_{\mathbb{A}^{(n)}_{\texttt{a}}}   R_{\texttt{a}}(\boldsymbol{\xi}) \rho_{\texttt{a}} (\boldsymbol{\xi})  \text{d} \boldsymbol{\xi}
=
\sum_{\lambda\in\Lambda^{(m,n)}_{\texttt{a}}}      R_{\texttt{a}}\bigl( \boldsymbol{\xi}^{(m,n)}_{\texttt{a};\lambda}  \bigr)   \rho_{\texttt{a}} \bigl( \boldsymbol{\xi}^{(m,n)}_{\texttt{a};\lambda}\bigr) \Delta^{(m,n)}_{\texttt{a};\lambda } 
\end{equation*}
with
$\Delta^{(m,n)}_{\texttt{a};\lambda }:=  | C_{\texttt{a}}(\boldsymbol{\xi}^{(m,n)}_{\texttt{a};\lambda} ;q) |^2\hat{\Delta}^{(m,n)}_{\texttt{a};\lambda }$ and $
R_{\texttt{a}}(\boldsymbol{\xi}):=   \frac{f(\boldsymbol{\xi})}{{O}_{\texttt{a}}(\boldsymbol{\xi};q) }
$, where the denominator is of the form
\begin{equation*}
O_{\texttt{a}}(\boldsymbol{\xi};q):=\prod_{1\leq j <k\leq n}  \bigl(1-2q\cos(\xi_j-\xi_k) +q^2\bigr)
\end{equation*}
and the numerator $f(\boldsymbol{\xi})$ is allowed to be any symmetric polynomial in $\mathbb{P}^{(2m-1,n)}_{\texttt{a}}$.
\end{subequations}
\end{corollary}
\begin{proof}
Immediate from Theorem \ref{hwc-a:thm}  via the identity  \begin{equation*}  | C_{\texttt{a}}(\boldsymbol{\xi} ;q) |^{-2} =  \rho_{\texttt{a}} (\boldsymbol{\xi}) /{O}_{\texttt{a}}(\boldsymbol{\xi};q) .  \end{equation*}
\end{proof}

\begin{remark}\label{non-minimal:rem}
When $n=2$, Theorem \ref{hwc-a:thm} boils down to a special instance  of  the quadrature rule on $m+1$ nodes presented  in \cite[Theorem 5]{die-ems:quadrature} (viz., with
$d=\tilde{d}=1$, $\epsilon_+=\tilde{\epsilon}_+=0$, $\epsilon_-=\tilde{\epsilon}_-=1$, and $\alpha_1=\tilde{\alpha}_1=-q$, respectively).
The corresponding degree of exactness is $\texttt{D}=2m-1$, which is off by two from the optimal Gaussian degree $2m+1$. More generally, via
a change of variables of the form  (cf. e.g.  \cite[Section 3.4]{mun:group}, \cite[Section 5.1]{li-xu:discrete} or \cite[Section 3]{hri-mot-pat:cubature}):
\begin{equation}\label{vchange}
X_j:=
\begin{cases}
\frac{1}{2}\left( M_{\texttt{a};\omega_j}(\boldsymbol{\xi})+ M_{\texttt{a};\omega_{n-j}}(\boldsymbol{\xi})\right) &\text{if}\
j=1,\ldots, \lfloor \frac{n-1}{2} \rfloor , \\
\frac{1}{\sqrt{2}}M_{\texttt{a};\omega_j}(\boldsymbol{\xi})&\text{if}\
j=\frac{n}{2} ,\\
\frac{1}{2i}\left( M_{\texttt{a};\omega_j}(\boldsymbol{\xi})-M_{\texttt{a};\omega_{n-j}}(\boldsymbol{\xi})\right)&\text{if}\
j=\lceil \frac{n+1}{2}\rceil ,\ldots ,n-1,
\end{cases}
\end{equation}
Theorem \ref{hwc-a:thm} can be reformulated as an exact cubature rule for  $f\in \Pi^{(2m-1,n-1)}$ supported on $\dim (\Pi^{(m,n-1)})$ nodes, where $\Pi^{(\texttt{D},n-1)}$  refers to the $\binom{\texttt{D}+n-1}{\texttt{D}}$-dimensional space of all polynomials in $X_1,\ldots ,X_{n-1}$ of total degree at most $\texttt{D}$:
\begin{align}
\frac{1}{\bigl(\pi \sqrt{2} \bigr)^{n-1} }\int_{\text{A}^{(n)}_{\texttt{a}} }  f(X_1,\ldots X_{n-1}) & \frac{\sqrt{\rho_{\texttt{a}}(X_1,\ldots,X_{n-1})}}{O_{\texttt{a}}(X_1,\ldots,X_{n-1};q)}  \text{d}X_1\cdots \text{d}X_{n-1} \\
&= \sum_{\lambda\in\Lambda^{(m,n)}_{\texttt{a}}}      f \bigl( \boldsymbol{X}^{(m,n)}_{\texttt{a};\lambda}  \bigr) \hat{\Delta}^{(m,n)}_{\texttt{a};\lambda } .\nonumber
\end{align}
Here $\rho_{\texttt{a}}$ and $O_{\texttt{a}}$ refer to the transformed functions expressed in the new coordinates $X_1,\ldots,X_{n-1}$, and
\begin{align*}
\text{A}^{(n)}_{\texttt{a}} & :=\left\{ \bigl(X_1(\boldsymbol{\xi}),\ldots ,X_{n-1}(\boldsymbol{\xi})\bigr) \mid \boldsymbol{\xi}\in\mathbb{A}^{(n)}_{\texttt{a}}\right\},\\
 \boldsymbol{X}^{(m,n)}_{\texttt{a};\lambda} &:= \bigl(X_1( \boldsymbol{\xi}^{(m,n)}_{\texttt{a};\lambda} ),\ldots ,X_{n-1}( \boldsymbol{\xi}^{(m,n)}_{\texttt{a};\lambda} )\bigr) .
 \end{align*}
 To perform this coordinate transformation one uses that on the hyperplane $\xi_1+\cdots+\xi_n=0$  
  the Jacobian is given by
 $\left| \frac{\partial (X_1,\ldots,X_{n-1})}{\partial (\xi_1,\ldots,\xi_{n-1})}\right| =\bigl(\frac{1}{\sqrt{2}}\bigr)^{n-1} \sqrt{\rho_{\texttt{a}}(\boldsymbol{\xi})}$  (cf. \cite[Lemma 5.5]{li-xu:discrete} or \cite[Proposition 4]{hri-mot-pat:cubature}) and the volume form reads
 $\text{d}\boldsymbol{\xi}=\sqrt{n}\,\text{d}\xi_1\cdots\text{d}\xi_{n-1}$.
 The linear isomorphism between
the spaces $\mathbb{P}^{(\texttt{D},n)}_{\texttt{a}}$  and
$\Pi^{(\texttt{D},n-1)}$ induced by this change of variables
reveals that the number of nodes to achieve the exact integration for all
$f\in \mathbb{P}^{(2m-1,n)}_{\texttt{a}}$ is bounded from below by the (Gaussian) value $\dim \bigl(\mathbb{P}^{(m-1,n)}_{\texttt{a}}\bigr)$ (cf. e.g. \cite[Chapter 3.8]{dun-xu:orthogonal}). The number of
nodes employed by
the cubature rule in Theorem \ref{hwc-a:thm} thus exceeds this lower bound by
\begin{equation*}
\dim \bigl(\mathbb{P}^{(m,n)}_{\texttt{a}}\bigr)-\dim \bigl(\mathbb{P}^{(m-1,n)}_{\texttt{a}}\bigr)=\binom{n+m-2}{m}=\dim \bigl(\mathbb{P}^{(m,n-1)}_{\texttt{a}}\bigr) .
\end{equation*}
 \end{remark}

\begin{remark}
The symmetric functions $R_{\texttt{a}}(\boldsymbol{\xi})=   \frac{f(\boldsymbol{\xi})}{{O}_{\texttt{a}}(\boldsymbol{\xi};q) } $  admit
simple poles supported on the zero locus of the denominator $O_{\texttt{a}}(\boldsymbol{\xi};q)$.  For $0<q<1$ the pole locus in question consists of the complex hyperplanes
\begin{equation*}
\xi_j-\xi_k = i\log (q) \mod 2\pi\qquad (1\leq j\neq k\leq n), 
\end{equation*}
whereas for $-1<q<0$ it consists of the complex hyperplanes 
\begin{equation*}
\xi_j-\xi_k =\pi+ i\log (-q) \mod 2\pi\qquad (1\leq j\neq k\leq n).
\end{equation*}
At the boundary of the parameter domain $-1<q<1$ this complex hyperplane arrangement approximates itself to the closure of the integration domain $\mathbb{A}^{(n)}_{\texttt{a}}$ \eqref{A:a}.
Indeed, 
for $q\to 1$ the pole locus intersects the closure of $\mathbb{A}^{(n)}_{\texttt{a}}$  at the boundary hyperplanes $\xi_j-\xi_{j+1}=0$ ($j=1,\ldots n-1$) and $\xi_1-\xi_n=2\pi$, while for  
$q\to -1$ the intersection stems from the hyperplanes passing through the interior: $\xi_j-\xi_k =\pi$ ($1\leq j<k\leq n$).
\end{remark}

\section{Cubature rules associated with hyperoctahedral Hall-Littlewood polynomials}\label{sec3}
In this  section the above construction is adapted for the compact symplectic group $Sp (n;\mathbb{H})$. The pertinent Haar measure corresponds  to the density of the
circular quaternion ensemble
\begin{equation}
\rho_{\texttt{b}}(\boldsymbol{\xi}):=2^{n(n+1)}\prod_{1\leq j\leq n}   \bigl(1-\cos^2(\xi_j)\bigr)  \prod_{1\leq j<k\leq n}  \bigl(\cos (\xi_j)-\cos(\xi_k)\bigr)^2 
\end{equation}
on the fundamental domain
\begin{equation}\label{A:b}
\mathbb{A}^{(n)}_{\texttt{b}} := \{  \boldsymbol{\xi}=(\xi_1,\ldots ,\xi_n) \in\mathbb{R}^n \mid \pi > \xi_1 >\xi_2>\cdots >\xi_n > 0\} .
\end{equation}
Macdonald's hyperoctahedral Hall-Littlewood polynomials produce in this situation cubature formulas for the integration
of symmetric functions in $z_j=\cos(\xi_j)$ ($j=1,\ldots ,n$)  over the fundamental domain $\mathbb{A}^{(n)}_{\texttt{b}} $ with respect to the density $\rho_{\texttt{b}}(\boldsymbol{\xi})$.

\subsection{Hyperoctahedral Hall-Littlewood polynomials}
Macdonald's hyperoctahedral Hall-Littlewood polynomials are a variant of the Hall-Littlewood polynomials associated with the hyperoctahedral group of signed permutations, which can be
retrieved from \cite[\S 10]{mac:orthogonal} upon picking the root system $R$ of type $BC_n$.
The polynomials in question are labeled by $Sp (n;\mathbb{H})$ dominant weight vectors
\begin{equation}\label{dominant-cone:b}
\Lambda^{(n)}_{\texttt{b}}:= \{   (\mu_1,\ldots ,\mu_n)\in\mathbb{Z}^n\mid  \mu_1\geq\cdots\geq\mu_n\geq 0\} 
\end{equation}
that are nonnegatively generated by the fundamental basis $e_1+\cdots +e_j$, $j=1,\ldots ,n$   (cf. \cite[Planche III]{bou:groupes}). 
Here we restrict attention to a two-parameter subfamily of these polynomials given by
\begin{subequations}
\begin{align}\label{HLp:b}
P_{\texttt{b};\mu} (\boldsymbol{\xi} ;q,q_0)& :=   \\
 \sum_{\substack{ \sigma\in S_n \\ \epsilon\in \{ 1,-1\}^n}}  & C_{\texttt{b}}(\epsilon_1 \xi_{\sigma_1},\ldots , \epsilon_n \xi_{\sigma_n};q,q_0)
\exp (i\epsilon_1 \xi_{\sigma_1}\mu_1+\cdots +i \epsilon_n \xi_{\sigma_n} \mu_n)  , \nonumber
\end{align}
with $\mu\in \Lambda^{(n)}_{\texttt{b}}$ and
\begin{eqnarray}\label{Cp:b}
\lefteqn{C_{\texttt{b}}(\xi_1,\ldots ,\xi_n;q,q_0)=C_{\texttt{b}}(\boldsymbol{\xi};q,q_0) :=}
 && \\
&& 
\prod_{1\leq j\leq n} \frac{1-q_0 e^{-i\xi_j}}{1-e^{-2i\xi_j}}
 \prod_{1\leq j<k \leq n} \left(\frac{1-q e^{-i(\xi_{j}-\xi_k)}}{1-e^{-i(\xi_{j}-\xi_k)}}\right)\left(  \frac{1-q e^{-i(\xi_{j}+\xi_k)}}{1-e^{-i(\xi_{j}+\xi_k)}} \right)  .\nonumber
\end{eqnarray}
\end{subequations}
The symmetrization is now with respect to the action of the hyperoctahedral group of signed permutations, which involves summing over  all
 $\sigma= { \bigl( \begin{smallmatrix}1& 2& \cdots & n \\
 \sigma_1&\sigma_2&\cdots & \sigma_n
 \end{smallmatrix}\bigr)}\in S_n$ and all
$\epsilon=(\epsilon_1,\ldots,\epsilon_n)\in \{ 1,-1\}^n$.  For $-1<q,q_0<1$ the polynomials $P_{\texttt{b};\mu} (\boldsymbol{\xi} ;q,q_0)$ satisfy the following orthogonality relations  \cite[\S 10]{mac:orthogonal}:
\begin{align}\label{or-c:b}
\frac{1}{(2\pi )^{n} }   \int_{\mathbb{A}^{(n)}_{\texttt{b}}}    P_{\texttt{b};\mu}  (\boldsymbol{\xi};q,q_0) \overline{P_{\texttt{b};\nu}  (\boldsymbol{\xi};q ,q_0) }    | C_{\texttt{b}}(\boldsymbol{\xi};q, q_0) |^{-2}  \text{d} 
\boldsymbol{\xi} &  \\
=
\begin{cases} 
 \prod_{\substack{1\leq j<k\leq n\\ \mu_j-\mu_k=0}}  \frac{1-q^{1+k-j}}{1-q^{k-j}} &\text{if}\  \nu= \mu ,\\
0 &\text{if}\  \nu\neq \mu 
\end{cases}
 & \nonumber
\end{align}
($\mu,\nu\in\Lambda^{(n)}_{\texttt{b}}$).

\begin{remark}
In the orthogonality relations \eqref{or-c:b}  the integration is meant with respect to the standard Lebesgue  measure $\text{d}\boldsymbol{\xi}=\text{d}\xi_1\cdots \text{d}\xi_n$ for $\mathbb{R}^n$. In particular:
$\int_{\mathbb{A}^{(n)}_{\texttt{b}}} \text{d}\boldsymbol{\xi} = \text{Vol} \bigl(\mathbb{A}^{(n)}_{\texttt{b}}\bigr) = {(\pi)^{n}}/{n!} $.
\end{remark}

\subsection{Finite-dimensional orthogonality relations}
In the same spirit as before, the construction of an appropriate lattice Laplacian (with Robin type boundary conditions) on the finite alcove
\begin{equation}\label{dwv:b}
\Lambda^{(m,n)}_{\texttt{b}}:= \{   (\mu_1,\ldots ,\mu_n)\in\mathbb{Z}^n\mid m\geq \mu_1\geq\cdots\geq\mu_n\geq 0\} 
\end{equation}
of level $m\in\mathbb{Z}_{> 0}$, has given rise to a novel finite-dimensional orthogonality relation for the hyperoctahedral Hall-Littlewood polynomials of the form \cite[Section 11.4]{die-ems:orthogonality}:
\begin{subequations}
\begin{equation}\label{d-ortho:b}
\sum_{\mu\in\Lambda^{(m,n)}_{\texttt{b}}}      P_{\texttt{b};\mu} \bigl( \boldsymbol{\xi}^{(m,n)}_{\texttt{b};\lambda} ;q, q_0 \bigr)     \overline{P_{\texttt{b};\mu} \bigl( \boldsymbol{\xi}^{(m,n)}_{\texttt{b};\kappa}  ;q, q_0 \bigr) }
 \delta^{(m,n)}_{\texttt{b};\mu }  (q)
=
0 \quad \text{if}\  \kappa \neq  \lambda
\end{equation}
($\lambda,\kappa \in \Lambda^{(m,n)}_{\texttt{b}}$), where
 \begin{equation}\label{w:b}
\delta^{(m,n)}_{\texttt{b};\mu }(q):= \prod_{\substack{1\leq j<k\leq n\\ \mu_j-\mu_k=0}}  \frac{1-q^{k-j}}{1-q^{1+k-j}} .
\end{equation}
\end{subequations}
In the current situation,
the positions of  the nodes $ \boldsymbol{\xi}^{(m,n)}_{\texttt{b};\lambda}$, $\lambda\in\Lambda^{(m,n)}_{\texttt{b}}$ parametrizing the eigenvalues of the lattice Laplacian turn out to depend on three
parameters in the interval $(-1,1)$: $q_0$ (which governs the boundary condition at the wall $\mu_n=0$), $q$ (which governs the boundary condition at the walls $\mu_j=\mu_{j+1}$, $j\in \{ 1,\ldots ,n-1\}$), and an additional parameter $q_1$ (which governs the boundary condition at the affine wall $\mu_1=m$) \cite{die-ems:orthogonality}.

\subsection{Positions of the nodes}
In the absence of explicit formulas for the positions of the nodes in Eqs. \eqref{d-ortho:b}, \eqref{w:b} at general values of the parameters $q,q_0,q_1\in (-1,1)$, we recur again to a characterization in terms of the minimum of an associated Morse function
from  \cite[Section 10.1]{die-ems:orthogonality} so as to enable numerical computations.

Specifically, the position of the node  $ \boldsymbol{\xi}^{(m,n)}_{\texttt{b};\lambda}$, $\lambda\in\Lambda^{(m,n)}_{\texttt{b}}$ turns out to be governed by the unique global minimum of the following semi-bounded Morse function
$V^{(m,n)}_{\texttt{b};\lambda} :\mathbb{R}^n\to\mathbb{R}$:
\begin{subequations}
\begin{multline}\label{morse-b}
V^{(m,n)}_{\texttt{b};\lambda}  (\boldsymbol{\xi}):= \sum_{1\le j < k \le n }
\left(
    \int_0^{\xi_j+\xi_k} v_q(\vartheta )\text{d}\vartheta+   \int_0^{\xi_j-\xi_k} v_q(\vartheta)\text{d}\vartheta 
\right) \\
+ \sum_{1\leq j\leq n} \left(
(m+1)\xi_j^2-2\pi (\varrho_{\texttt{b};j}+\lambda_j)\xi_j
+  \int^{\xi_j}_0 \bigl(v_{q_0}(\vartheta )+v_{q_1}(\vartheta )\bigr)\text{d}\vartheta \right) ,
\end{multline}
where $v_{q} (\vartheta)$ is of the form in Eq. \eqref{uv},  and
\begin{equation}
\varrho_{\texttt{b};j} :=n+1-j\quad  (j=1,\ldots ,n).
\end{equation}
\end{subequations} 
As before, the existence and uniqueness of the global minimum $ \boldsymbol{\xi}^{(m,n)}_{\texttt{b};\lambda}$ of $V^{(m,n)}_{\texttt{b};\lambda}  (\boldsymbol{\xi})$ is 
ensured by the unbounded radial growth $V^{(m,n)}_{\texttt{b};\lambda}  (\boldsymbol{\xi})\to +\infty$ for $|\boldsymbol{\xi}|\to\infty$, in combination with the convexity:
\begin{align}\label{Hesse:b}
&H^{(n,m)}_{\texttt{b};j,k}:=\partial_{\xi_j}\partial_{\xi_k} V^{(n,m)}_{ \texttt{b};\lambda} (\boldsymbol{\xi}) \\
&=
\begin{cases}
2(m+1) + u_{q_0} (\xi_j) + u_{q_1}(\xi_j) + \sum_{\substack{1\leq l\leq n \\l\neq j}}\bigl (u_q(\xi_j+\xi_l) +u_q(\xi_j-\xi_l) \bigr) & \text{if $ k=j$}\\
u_q(\xi_j+\xi_k) -u_q(\xi_j-\xi_k) & \text{if $k\neq j$}\\
\end{cases} \nonumber
\end{align}
(with $u_q(\theta)$ taken from  Eq. \eqref{uv}), so
\begin{align*}
\sum_{1\leq j,k\leq n}  H^{(m,n)}_{\texttt{b};j,k} x_j x_k   
= & \sum_{1\leq j\leq n} \Bigl( 2(m+1)+ u_{q_0}(\xi_j) + u_{q_1}(\xi_j)\Bigr) x_j^2 \\
&+  \sum_{1\leq j<k\leq n} \Bigl( u_q(\xi_j +\xi_k)(x_j+x_k)^2+
   u_q(\xi_j -\xi_k)(x_j-x_k)^2\Bigr) \\
 \ge & 2(m+1) \sum_{1\leq j\leq n} x_j^2.
\end{align*}

The equations $ \partial_{\xi_j}V^{(m,n)}_{\texttt{b};\lambda}(\boldsymbol{\xi})=0$ for the numerical computation of  the position of the node $ \boldsymbol{\xi}^{(m,n)}_{\texttt{b};\lambda}$ using Newton's method now become:
\begin{equation}\label{CEQ:b}
2(m+1)\xi_j+v_{q_0}(\xi_j)+v_{q_1}(\xi_j)+\sum_{\substack{1\leq k\leq n\\k\neq j}} \Bigl( v_q(\xi_j+\xi_k)+ v_q(\xi_j-\xi_k)\Bigr)=2\pi (\varrho_{\texttt{b};j}+\lambda_j) ,
\end{equation}
for $j=1,\ldots,n$.  A corresponding suitable initial estimate stemming from the explicit solution of this system at the linear point $(q,q_0,q_1)=(0,0,0)$ in parameter space is given by 
$\xi_j=\frac{\pi (\lambda_j+\varrho_{\texttt{b};j})}{n+m+1}$, $j=1,\ldots,n$ (cf Remark \ref{bounds-b:rem} below).

\begin{remark}\label{bounds-b:rem}
Upon adapting Remarks \ref{regularity-a:rem} and \ref{bounds-a:rem}, it is deduced from Eq.  \eqref{CEQ:b} that $\boldsymbol{\xi}^{(m,n)}_{\texttt{b};\lambda}\in\mathbb{A}^{(n)}_{\texttt{b}}$ \eqref{A:b} for any $\lambda\in\Lambda^{(m,n)}_{\texttt{b}}$ \eqref{dwv:b}, and that at
$\boldsymbol{\xi}=\boldsymbol{\xi}^{(m,n)}_{\texttt{b};\lambda}$:
 \begin{subequations}
\begin{equation}\label{gap-b:a}
\frac{\pi(n+1-j+\lambda_j)}{m+1+\kappa_{\texttt{b};-}(q,q_0,q_1)} \leq \xi_j \leq \frac{\pi(n+1-j+\lambda_j)}{m+1+\kappa_{\texttt{b};+}(q,q_0,q_1)}
\end{equation}
for $1\leq j\leq n$,  and
\begin{equation}\label{gap-b:b}
\frac{\pi(k-j+\lambda_j-\lambda_k)}{m+1+\kappa_{\texttt{b};-}(q,q_0,q_1)} \leq \xi_j-\xi_k\leq \frac{\pi(k-j+\lambda_j-\lambda_k)}{m+1+\kappa_{\texttt{b};+}(q,q_0,q_1)} 
\end{equation}
for $1 \le  j < k \le n$, where
\begin{align}\label{gap-b:c}
\kappa_{\texttt{b};\pm}(q,q_0,q_1):=&\frac{1}{2}\left( \frac{1-|q_0|}{1 + |q_0|} \right)^{\pm 1}+\frac{1}{2} \left(\frac{1-|q_1|}{1+ |q_1|} \right)^{\pm 1}  \\
&+ (n-1)\left( \frac{1-|q|}{1+|q|}\right)^{\pm 1} . \nonumber
\end{align}
\end{subequations}
These bounds confirm that for vanishing parameter values:
\begin{equation}
\left.  \boldsymbol{\xi}^{(m,n)}_{\texttt{b};\lambda}\right|_{q,q_0,q_1=0}=\frac{\pi (\lambda+\varrho_{\texttt{b}})}{n+m+1}\qquad (\lambda\in\Lambda^{(m,n)}_{\texttt{b}}),
\end{equation}
where $\varrho_{\texttt{b}}:=(\varrho_{\texttt{b},1},\ldots , \varrho_{\texttt{b},n})$.
Moreover, since at $\boldsymbol{\xi}=\frac{\pi (\lambda+\varrho_{\texttt{b}})}{n+m+1}$ ($\lambda\in\Lambda^{(m,n)}_{\texttt{b}}$) the inequalities in Eqs. \eqref{gap-b:a}--\eqref{gap-b:c} are satisfied for any
$-1<q,q_0,q_1<1$, this special point provides a convenient initial estimate when computing the position of the node $\boldsymbol{\xi}^{(m,n)}_{\texttt{b};\lambda}$ numerically from Eq. \eqref{CEQ:b} via Newton's method.
\end{remark}

\subsection{Cubature rule}\label{cub:b:sec}
Let
\begin{equation}
\mathbb{P}^{(m,n)}_{\texttt{b}}:=  \text{Span}_{\mu\in\Lambda^{(m,n)}_{\texttt{b}}}   \{    M_{\texttt{b};\mu} (\boldsymbol{\xi})  \}  ,
\end{equation}
with
\begin{subequations}
\begin{equation}\label{smonomials:b}
M_{\texttt{b};\mu} (\boldsymbol{\xi}) :=  \frac{1}{N_{\texttt{b}; \mu}}   \sum_ {\substack{\sigma\in S_{n} \\ \epsilon\in \{ 1,-1\}^n}}   
\exp (i \epsilon_1 \xi_{\sigma_1}\mu_1+\cdots +i \epsilon_n \xi_{\sigma_{n}} \mu_{n})  
\end{equation}
 normalized such that each exponential term on the RHS occurs with multiplicity one:
\begin{equation}\label{stabilizer:b}
  N_{\texttt{b};\mu} := 2^{\text{m}_0(\mu)} \prod_{\substack{1\leq j<k\leq n\\ \mu_j-\mu_k=0}}  \frac{1+k-j}{k-j}    
  \end{equation}
\end{subequations}
(cf. Eq. \eqref{mult}). Notice that $\mathbb{P}^{(m,n)}_{\texttt{b}}$ amounts to the
$\binom{m+n}{m}$-dimensional space of symmetric polynomials of degree at most $m$ in each of the variables $z_j=\cos (\xi_j)$ ($j\in \{ 1,\ldots ,n\}$). 

\begin{theorem}[Hyperoctahedral Hall-Littlewood Cubature]\label{hwc-b:thm}
For $q,q_0,q_1\in (-1,1)$ and $m\in\mathbb{Z}_{> 0}$, the following cubature rule holds true for any  symmetric polynomial
$f(\boldsymbol{\xi})$  in $\mathbb{P}^{(2m,n)}_{\texttt{b}}$:
\begin{subequations}
\begin{equation}\label{hwc:b}
\frac{1}{(2\pi )^{n} }    \int_{\mathbb{A}^{(n)}_{\texttt{b}}}   f(\boldsymbol{\xi})  | C_{\texttt{b}}(\boldsymbol{\xi};q,q_0) |^{-2} 
  \text{d} \boldsymbol{\xi}
=
\sum_{\lambda\in\Lambda^{(m,n)}_{\texttt{b}}}      f \bigl( \boldsymbol{\xi}^{(m,n)}_{\texttt{b};\lambda}  \bigr)   \hat{\Delta}^{(m,n)}_{\texttt{b};\lambda } ,
\end{equation}
with Christoffel weights given by
\begin{equation}\label{cw:b}
\hat{\Delta}^{(m,n)}_{\texttt{b};\lambda }:= 
\Biggl(\sum_{\mu\in \Lambda^{(m,n)}_{\texttt{b}}}   \left| P_{\texttt{b};\mu} \bigl(\boldsymbol{\xi}^{(m,n)}_{\texttt{b};\lambda} ;q,q_0\bigr) \right|^2  \delta^{(m,n)}_{\texttt{b};\mu } (q) \Biggr)^{-1}  .
\end{equation}
\end{subequations}
\end{theorem}

 \begin{proof}
It is immediate from the discrete orthogonality relations in Eqs. \eqref{d-ortho:b}, \eqref{w:b}  that the  following matrix is unitary:
\begin{equation*}
\left[ 
\sqrt{ \delta^{(m,n)}_{\texttt{b};\mu }(q)} P_{\texttt{b};\mu} \bigl( \boldsymbol{\xi}^{(m,n)}_{\texttt{b};\lambda} ;q,q_0\bigr)  \sqrt{\hat{\Delta}^{(m,n)}_{\texttt{b};\lambda }}
\right]_{\mu,\lambda\in \Lambda^{(m,n)}_{\texttt{b}}}  .
\end{equation*}
By `column-row duality' this means that for any $\mu,\nu\in \Lambda^{(m,n)}_{\texttt{b}}$:
\begin{align}\label{or-d:b}
\sum_{\lambda\in\Lambda^{(m,n)}_{\texttt{b}}}      P_{\texttt{b};\mu} \bigl( \boldsymbol{\xi}^{(m,n)}_{\texttt{b};\lambda} ;q, q_0 \bigr)    
\overline{  P_{\texttt{b};\nu} \bigl( \boldsymbol{\xi}^{(m,n)}_{\texttt{b};\lambda}  ;q, q_0 \bigr) } 
\hat{\Delta}^{(m,n)}_{\texttt{b};\lambda }  & \\
=
\begin{cases} 
1/ \delta^{(m,n)}_{\texttt{b};\mu }(q) &\text{if}\  \nu= \mu , \\
0 &\text{if}\  \nu\neq \mu .
\end{cases} & \nonumber
\end{align}
Upon comparing with the standard orthogonality relations for the corresponding hyperoctahedral Hall-Littlewood polynomials in Eq. \eqref{or-c:b},
 it is seen that the asserted cubature rule is valid for all symmetric polynomials $f(\boldsymbol{\xi})$ of the form
\begin{equation}\label{kostka:b}
f(\boldsymbol{\xi})=     P_{\texttt{b};\mu}  (\boldsymbol{\xi};q,q_0) P_{\texttt{b};\nu}  (\boldsymbol{\xi};q,q_0 )   \quad\text{with}\quad \mu,\nu\in  \Lambda^{(m,n)}_{\texttt{b}}
\end{equation}
(where we have used that $\overline{P_{\texttt{b};\nu}  (\boldsymbol{\xi};q,q_0 )}=P_{\texttt{b};\nu}  (\boldsymbol{\xi};q,q_0 )$).
Since the products in question span $\mathbb{P}^{(2m,n)}_{\texttt{b}}$
(because the monomial expansion of $f(\boldsymbol{\xi})$ \eqref{kostka:b} contains $M_{\texttt{b};\mu+\nu}(\boldsymbol{\xi})$ and symmetric monomials $M_{\texttt{b};\kappa}(\boldsymbol{\xi})$ with $\kappa$ smaller than $\mu+\nu$ in the
dominance partial order),
the cubature rule again  follows for general symmetric polynomials $f(\boldsymbol{\xi})$ in $\mathbb{P}^{(2m,n)}_{\texttt{b}}$ by linearity.
\end{proof}

Theorem \ref{hwc-b:thm} can be reinterpreted in turn as an exact cubature rule for the integration of a class of rational symmetric functions against the density of the
circular quaternion ensemble (CQE).

\begin{corollary}[Cubature in CQE]\label{hwc-b:cor}
For $q,q_0,q_1\in (-1,1)$ and $m\in\mathbb{Z}_{> 0}$, one has that
\begin{subequations}
\begin{equation*}\label{hwrc:b}
\frac{1}{(2\pi )^{n} }   \int_{\mathbb{A}^{(n)}_{\texttt{b}}}   R_{\texttt{b}}(\boldsymbol{\xi}) \rho_{\texttt{b}} (\boldsymbol{\xi})  \text{d} \boldsymbol{\xi}
=
\sum_{\lambda\in\Lambda^{(m,n)}_{\texttt{b}}}      R_{\texttt{b}}\bigl( \boldsymbol{\xi}^{(m,n)}_{\texttt{b};\lambda}  \bigr)   \rho_{\texttt{b}} \bigl( \boldsymbol{\xi}^{(m,n)}_{\texttt{b};\lambda}\bigr) \Delta^{(m,n)}_{\texttt{b};\lambda } 
\end{equation*}
with
$ \Delta^{(m,n)}_{\texttt{b};\lambda } :=    | C_{\texttt{b}}(\boldsymbol{\xi};q,q_0) |^{2} \hat{ \Delta}^{(m,n)}_{\texttt{b};\lambda } $ and
$
R_{\texttt{b}}(\boldsymbol{\xi}):=   \frac{f(\boldsymbol{\xi})}{O_{\texttt{b}}(\boldsymbol{\xi};q,q_0)} $, where the denominator is of the form
\begin{align*}
O_{\texttt{b}}(\boldsymbol{\xi};q,q_0):=\prod_{1\leq j <k\leq n} &  (1-2q\cos(\xi_j-\xi_k) +q^2)  (1-2q\cos(\xi_j+\xi_k) +q^2) \nonumber \\
&\times  \prod_{1\leq j\leq n}  (1-2q_0\cos(\xi_j) +q_0^2) .
\end{align*}
and the numerator $f(\boldsymbol{\xi})$ is allowed to be any symmetric polynomial in $\mathbb{P}^{(2m,n)}_{\texttt{b}}$.
\end{subequations}
\end{corollary}
\begin{proof}
Immediate from Theorem \ref{hwc-b:thm}  via the identity  \begin{equation*} | C_{\texttt{b}}(\boldsymbol{\xi};q,q_0) |^{-2} =   \rho_{\texttt{b}} (\boldsymbol{\xi}) / O_{\texttt{b}}(\boldsymbol{\xi};q,q_0) . \end{equation*}
\end{proof}

\begin{remark}\label{n=1-b:rem}
For $n=1$, Theorem \ref{hwc-b:thm} recovers another special instance  of  the quadrature rule on $m+1$ nodes presented  in \cite[Theorem 5]{die-ems:quadrature} (viz., with
$d=\tilde{d}=1$, $\epsilon_\pm=\tilde{\epsilon}_\pm=1$, and $\alpha_1=-q_0$, $\tilde{\alpha}_1=-q_1$, respectively).
Its degree of exactness $\texttt{D}=2m$ is only one shy of the optimal Gaussian degree $2m+1$. For general $n$ we can proceed as in Remark \ref{non-minimal:rem} and
perform a  change of variables of the form
\begin{equation}
X_j:=M_{\texttt{b}; e_1+\cdots +e_j}(\boldsymbol{\xi}), \quad j=1,\ldots ,n
\end{equation}
(cf.  e.g. \cite[Section 7]{hof-wit:generalized}, \cite[Section 3]{moo-pat:cubature} and \cite[Section 3]{hri-mot-pat:cubature}),  so as to recast Theorem \ref{hwc-b:thm}  in the form of an exact (Gaussian)
cubature rule for  $f\in\Pi^{(2m,n)}$ supported on $\dim (\Pi^{(m,n)})$ nodes:
\begin{align}
\frac{1}{(2\pi )^{n} }  \int_{\text{A}^{(n)}_{\texttt{b}} }  f(X_1,\ldots X_n) & \frac{\sqrt{\rho_{\texttt{b}}(X_1,\ldots,X_n)}}{O_{\texttt{b}}(X_1,\ldots,X_n;q,q_0)}  \text{d}X_1\cdots \text{d}X_n \\
&= \sum_{\lambda\in\Lambda^{(m,n)}_{\texttt{b}}}      f \bigl( \boldsymbol{X}^{(m,n)}_{\texttt{b};\lambda}  \bigr) \hat{\Delta}^{(m,n)}_{\texttt{b};\lambda } ,\nonumber
\end{align}
where $\rho_{\texttt{b}}$ and $O_{\texttt{b}}$ refer to the transformed functions expressed in the new coordinates $X_1,\ldots,X_n$, and
\begin{align*}
\text{A}^{(n)}_{\texttt{b}} & :=\left\{ \bigl(X_1(\boldsymbol{\xi}),\ldots ,X_n(\boldsymbol{\xi})\bigr) \mid \boldsymbol{\xi}\in\mathbb{A}^{(n)}_{\texttt{b}}\right\},\\
 \boldsymbol{X}^{(m,n)}_{\texttt{b};\lambda} &:= \bigl(X_1( \boldsymbol{\xi}^{(m,n)}_{\texttt{b};\lambda} ),\ldots ,X_n( \boldsymbol{\xi}^{(m,n)}_{\texttt{b};\lambda} )\bigr) .
 \end{align*}
Here we have used that the Jacobian is now of the form  $\left| \frac{\partial (X_1,\ldots,X_n)}{\partial (\xi_1,\ldots,\xi_n)}\right| =\sqrt{\rho_{\texttt{b}}(\boldsymbol{\xi})}$ (cf. e.g. \cite[Proposition 3.3]{moo-pat:cubature} and  \cite[Proposition 4]{hri-mot-pat:cubature}). The upshot is that in the present situation it is seen from this change of variables that the number of nodes employed by the cubature rule in Theorem \ref{hwc-b:thm}---achieving the exact integration for all
$f\in\mathbb{P}^{(2m,n)}_{\texttt{b}}$---coincides with the (Gaussian) lower bound $\dim (\mathbb{P}^{(m,n)}_{\texttt{b}})$ (cf. e.g. \cite[Chapter 3.8]{dun-xu:orthogonal}).
\end{remark}

\begin{remark}
The locus where the symmetric functions $
R_{\texttt{b}}(\boldsymbol{\xi}):=   \frac{f(\boldsymbol{\xi})}{O_{\texttt{b}}(\boldsymbol{\xi};q,q_0)}$ admit
simple poles stemming from the denominator $O_{\texttt{b}}(\boldsymbol{\xi};q,q_0)$ is given by  the following complex hyperplane arrangement:
\begin{align*}
\xi_j-\xi_k &=  \frac{\pi}{2}\bigl(1-\text{sign}(q) \bigr) \pm  i\log (| q |) \mod 2\pi \\
\xi_j+\xi_k &=  \frac{\pi}{2}\bigl(1-\text{sign}(q) \bigr) \pm  i\log (| q |) \mod 2\pi
\end{align*}
$(1\leq j< k\leq n)$, and
\begin{equation*}
\xi_j =  \frac{\pi}{2}\bigl(1-\text{sign}(q_0) \bigr) \pm  i\log (| q_0 |) \mod 2\pi\
\end{equation*}
$(1\leq j \leq n)$. At the boundary of the parameter domain $-1< q,q_0 < 1$ this pole locus approximates the closure of the integration domain $\mathbb{A}^{(n)}_{\texttt{b}}$ \eqref{A:b}  
via: (i) the boundary hyperplanes $\xi_j-\xi_{j+1}=0$ ($j=1,\ldots n-1$) when $q\to 1$, (ii) the boundary hyperplane $\xi_n=0$ when $q_0\to 1$,
(iii) the hyperplanes passing through the interior: $\xi_j\pm \xi_k =\pi$ ($1\leq j<k\leq n$) when $q\to -1$, and  (iv) the boundary hyperplane $\xi_1=\pi$ when $q_0\to -1$.
\end{remark}

\section{The positions of the (hyperoctahedral) Hall-Littlewood cubature nodes as roots of associated quasi-orthogonal polynomials}\label{sec4}
In this section the cubature nodes are shown to be common roots of associated quasi-orthogonal polynomials in $\mathbb{P}_{\texttt{c}}^{(m+1,n)}\setminus \mathbb{P}_{\texttt{c}}^{(m,n)}$ (where $\texttt{c}=\texttt{a}$ or $\texttt{c}=\texttt{b}$), cf. Proposition \ref{roots:prp} (below).

\subsection{Statement of the result}
For $\mu\in\Lambda^{(m+1,n)}_{\texttt{a}}\setminus \Lambda^{(m,n)}_{\texttt{a}}$ let
\begin{subequations}
\begin{equation}
 Q_{\texttt{a};\mu}  (\boldsymbol{\xi}) :=  P_{\texttt{a};\mu}  (\boldsymbol{\xi};q) - q^{\text{m}_{\mu_1}(\mu)  \text{m}_{\mu_n}(\mu)}  P_{\texttt{a};\mu-\omega_{\texttt{a};\mu}}  (\boldsymbol{\xi};q) 
\end{equation}
with
\begin{equation}
\omega_{\texttt{a};\mu} := \sum_{j=1}^{ \min ( \text{m}_{\mu_1}(\mu), \text{m}_{\mu_n}(\mu) )}  \bigl( e_{\text{m}_{\mu_1}(\mu) + 1-j}  -e_{n-\text{m}_{\mu_n}(\mu)+j} \bigr) ,
\end{equation}
\end{subequations}
and for $\mu\in\Lambda^{(m+1,n)}_{\texttt{b}}\setminus\Lambda^{(m,n)}_{\texttt{b}}$ let
\begin{subequations}
\begin{align}
 Q_{\texttt{b};\mu}  (\boldsymbol{\xi}) :=  &P_{\texttt{b};\mu}  (\boldsymbol{\xi};q,q_0) \\
 & -  q^{\frac{1}{2}\text{m}_{\mu_1}(\mu) ( \text{m}_{\mu_1}(\mu)-1)}  q_1^{\text{m}_{\mu_1}(\mu)} P_{\texttt{b};\mu-\omega_{\texttt{b};\mu}}  (\boldsymbol{\xi};q,q_0) \nonumber
\end{align}
with
\begin{equation}
\omega_{\texttt{b};\mu} :=e_1+e_2+\cdots + e_{\text{m}_{\mu_1}(\mu)} .
\end{equation}
\end{subequations}

These definitions ensure that $\mu-\omega_{\texttt{c};\mu}\in\Lambda_{\texttt{c}}^{(m,n)}$ and that $ Q_{\texttt{c};\mu} (\boldsymbol{\xi})$ enjoys the following quasi-orthogonality property: the polynomial in question is orthogonal---with respect to the inner products \eqref{or-c:a} and \eqref{or-c:b}, respectively---to the finite-dimensional subspace  of $\mathbb{P}^{(m,n)}_{\texttt{c}}$ spanned by the monomials $M_{\texttt{c};\nu}(\boldsymbol{\xi})$
with $\nu\in\Lambda_{\texttt{c}}^{(m,n)}$ smaller than $\mu-\omega_{\texttt{c};\mu}$ in the dominance partial order ($\texttt{c}\in\{\texttt{a},\texttt{b}\}$).

\begin{proposition}[Cubature Nodes as Roots  of $ Q_{\texttt{c};\mu} (\boldsymbol{\xi}) $]\label{roots:prp}
For $\texttt{c}\in\{ \texttt{a},\texttt{b}\}$,
the (hyperoctahedral) Hall-Littlewood cubature nodes $\boldsymbol{\xi}^{(m,n)}_{\texttt{c};\lambda}$, $\lambda\in\Lambda^{(m,n)}_{\texttt{c}}$ are common roots of the $\binom{m+n_{\texttt{c}}}{m+1}$ polynomials $Q_{\texttt{c};\mu}(\boldsymbol{\xi})$, $\mu\in\Lambda^{(m+1,n)}_{\texttt{c}}\setminus\Lambda^{(m,n)}_{\texttt{c}}$, where $n_{\texttt{a}}:=n-1$ and $n_{\texttt{b}}:=n$.
\end{proposition}

\begin{remark}\label{jump:rem} It follows from Proposition \ref{roots:prp} that at
$q_1=0$, one has that $ Q_{\texttt{b};\mu}  (\boldsymbol{\xi}) =  P_{\texttt{b};\mu}  (\boldsymbol{\xi};q,q_0) $ for $\mu\in\Lambda^{(m+1,n)}_{\texttt{b}}\setminus\Lambda^{(m,n)}_{\texttt{b}}$. 
Hence, the cubature nodes are in this situation common roots of the hyperoctahedral Hall-Littlewood polynomials
$P_{\texttt{b};\mu}  (\boldsymbol{\xi};q,q_0) $, $\mu\in\Lambda^{(m+1,n)}_{\texttt{b}}\setminus\Lambda^{(m,n)}_{\texttt{b}}$. The equality between
the inner products in Eqs. \eqref{or-c:b} and \eqref{or-d:b} then extends to all $ \mu\in\Lambda^{(m+1,n)}_{\texttt{b}}$ and $\nu\in\Lambda^{(m,n)}_{\texttt{b}}$. The upshot is that at $q_1=0$
the  cubature rule in Theorem \ref{hwc-b:thm} is valid for all $f(\boldsymbol{\xi})$ in $\mathbb{P}_{\texttt{b}}^{2m+1,n}$, i.e. the degree of exactness then jumps to the optimal Gaussian value $2m+1$.
\end{remark}

\subsection{Proof of Proposition \ref{roots:prp} for $\texttt{c}=\texttt{a}$}
To derive the proposition it is convenient to temporarily extend the definition of
$P_{\texttt{a};\mu}  (\boldsymbol{\xi};q)$ \eqref{HLp:a}, \eqref{Cp:a} to any
$\mu\in \text{Span}_{\mathbb{Z}}(\omega_1,\ldots, \omega_{n-1})$.  The corresponding Hall-Littlewood polynomials are known to obey the following straightening rule (cf. e.g.
\cite[Chapter III \S 2, Example 2]{mac:symmetric}).

\begin{lemma}[Straightening Rule]\label{s-rule-a:lem} For any $\mu\in \text{Span}_{\mathbb{Z}}(\omega_1,\ldots, \omega_{n-1})$, one has that
\begin{equation*}
\mu_j-\mu_{j+1}=-1  \Longrightarrow  P_{\texttt{a};\mu}  (\boldsymbol{\xi};q) =q  P_{\texttt{a};\mu  +e_{j}-e_{j+1} }  (\boldsymbol{\xi};q) \quad (j\in \{ 1,\ldots , n-1\}) .
\end{equation*}
\end{lemma}

\begin{proof}
For $j\in \{ 1,\ldots ,n-1\}$, let $r_j$ act on the components of $\boldsymbol{\xi}=(\xi_1,\ldots ,\xi_n)$ by transposing $\xi_j$ and $\xi_{j+1}$:
\begin{equation*}
r_j(\xi_1,\ldots ,\xi_n):= (\xi_1,\ldots,\xi_{j-1},\xi_{j+1},\xi_j,\xi_{j+2},\ldots ,\xi_n).
\end{equation*}
Then
\begin{equation*}
\frac{C_{\texttt{a}}(\boldsymbol{\xi};q)}{C_{\texttt{a}}(r_j\boldsymbol{\xi};q)} =\frac{q-e^{i(\xi_j-\xi_{j+1})}}{1-qe^{i(\xi_j-\xi_{j+1})}} ,
\end{equation*}
and thus
\begin{equation*}
C_{\texttt{a}}(\boldsymbol{\xi};q) + C_{\texttt{a}}(r_j\boldsymbol{\xi};q) e^{i(\xi_j-\xi_{j+1})} = qC_{\texttt{a}}(\boldsymbol{\xi};q) e^{i(\xi_j-\xi_{j+1})}+ qC_{\texttt{a}}(r_j\boldsymbol{\xi};q) .
\end{equation*}
Multiplication of both sides of the latter identity by $\exp ( i\xi_1\mu_1+\cdots+i\xi_n\mu_n)$, with  $\mu\in \text{Span}_{\mathbb{Z}}(\omega_1,\ldots, \omega_{n-1})$ such that
$\mu_j-\mu_{j+1}=-1$ (so $r_j\mu=\mu +e_j-e_{j+1}$), gives rise to the asserted straightening rule upon symmetrization with respect to the permutation action (on $\boldsymbol{\xi}$).
\end{proof}

Moreover, at the cubature nodes $\boldsymbol{\xi}=\boldsymbol{\xi}^{(m,n)}_{\texttt{a};\lambda}$, $\lambda\in\Lambda^{(m,n)}_{\texttt{a}}$ a system of algebraic relations
between the variables $\xi_1,\ldots ,\xi_n$ is satisfied:
\begin{equation}\label{BAE:a}
e^{im\xi_j} =(-1)^{n-1}\prod_{\substack{ 1\leq k\leq n \\ k\neq j}}  \left( \frac{1-q e^{i(\xi_j-\xi_k)}}{e^{i(\xi_j-\xi_k)}-q}\right) \qquad (j=1,\ldots ,n).
\end{equation}
Indeed, Eq. \eqref{BAE:a} is immediate from Eq. \eqref{CEQ:a} after multiplying by the imaginary unit and exponentiation of both sides with the aid of the identity
\begin{equation*}\label{identity}
\exp \bigl( - i v_q ( \vartheta ) \bigr)=  \left( \frac{1-q e^{i\vartheta }}{e^{i\vartheta }-q  }\right) \qquad (-1<q<1).
\end{equation*}
In this situation the Hall-Littlewood polynomials turn out to obey   an additional affine analogue of the above straightening rule (cf. Ref. \cite{die:diagonalization}).

\begin{lemma}[Affine Straightening Rule]\label{as-rule-a:lem}  For any $\mu\in \text{Span}_{\mathbb{Z}}(\omega_1,\ldots, \omega_{n-1})$ and variables $\boldsymbol{\xi}$ satisfying Eq. \eqref{BAE:a}, one has that
\begin{equation*}
\mu_1-\mu_{n}=m+1  \Longrightarrow  P_{\texttt{a};\mu}  (\boldsymbol{\xi};q) =q  P_{\texttt{a};\mu  -e_{1}+e_{n}}  (\boldsymbol{\xi};q) .
\end{equation*}
\end{lemma}

\begin{proof}
Let $r_{\texttt{a}}$ act on the components of $\boldsymbol{\xi}=(\xi_1,\ldots ,\xi_n)$ by transposition of $\xi_1$ and $\xi_{n}$:
\begin{equation*}
r_{\texttt{a}}(\xi_1,\ldots ,\xi_n):= (\xi_n,\xi_2,\ldots ,\xi_{n-1},\xi_1).
\end{equation*}
Then Eq. \eqref{BAE:a} implies that
\begin{equation*}
e^{im(\xi_1-\xi_n)}= \frac{C_{\texttt{a}}(r_{\texttt{a}}\boldsymbol{\xi};q)}{C_{\texttt{a}}(\boldsymbol{\xi};q)} \left( \frac{1-q e^{i(\xi_1-\xi_{n})}}{q-e^{i(\xi_1-\xi_{n})}} \right) ,
\end{equation*}
and thus
\begin{equation*}
C_{\texttt{a}}(\boldsymbol{\xi};q) + C_{\texttt{a}}(r_{\texttt{a}}\boldsymbol{\xi};q) e^{-i(m+1)(\xi_1-\xi_{n})} = qC_{\texttt{a}}(\boldsymbol{\xi};q) e^{-i(\xi_1-\xi_{n})}+ qC_{\texttt{a}}(r_{\texttt{a}}\boldsymbol{\xi};q)e^{-im(\xi_1-\xi_n)} .
\end{equation*}
Multiplication of both sides of the latter identity by $\exp ( i\xi_1\mu_1+\cdots+i\xi_n\mu_n)$, with  $\mu\in \text{Span}_{\mathbb{Z}}(\omega_1,\ldots, \omega_{n-1})$ such that
$\mu_1-\mu_{n}=m+1$ (so  $r_{\texttt{a}}\mu=\mu -(m+1)(e_1-e_{n})$), gives rise to the asserted affine straightening rule upon symmetrization with respect to the permutation action (on $\boldsymbol{\xi}$).
\end{proof}

Proposition \ref{roots:prp} now follows by iterated use of the straightening rules in Lemmas \ref{s-rule-a:lem} and \ref{as-rule-a:lem}. Indeed, if we first apply the affine straightening rule in Lemma  \ref{as-rule-a:lem}
to  $P_{\texttt{a};\mu}  (\boldsymbol{\xi};q)$ with $\mu\in\Lambda^{(m+1,n)}_{\texttt{a}}\setminus\Lambda^{(m,n)}_{\texttt{a}}$, and subsequently rearrange the components of $\mu -e_1+e_n$ in weakly decreasing order through iterated
transpositions employing the straightening rule of Lemma  \ref{s-rule-a:lem}, then it is readily seen that at $\boldsymbol{\xi}=\boldsymbol{\xi}^{(m,n)}_{\texttt{a};\lambda}$, $\lambda\in\Lambda^{(m,n)}_{\texttt{a}}$:
\begin{equation*}
P_{\texttt{a};\mu}  (\boldsymbol{\xi};q) = q^{\text{m}_{\mu_1}(\mu)  + \text{m}_{\mu_n}(\mu)-1}  P_{\texttt{a};\mu  -e_{\text{m}_{\mu_1}(\mu)}+e_{n-\text{m}_{\mu_n}(\mu) +1} } (\boldsymbol{\xi};q) .
\end{equation*}
Iteration of the latter relation entails that at $\boldsymbol{\xi}=\boldsymbol{\xi}^{(m,n)}_{\texttt{a};\lambda}$, $\lambda\in\Lambda^{(m,n)}_{\texttt{a}}$:
\begin{equation*}
P_{\texttt{a};\mu}  (\boldsymbol{\xi};q) = q^{\text{m}_{\mu_1}(\mu)  \text{m}_{\mu_n}(\mu)}  P_{\texttt{a};\mu-\omega_{\texttt{a};\mu}}  (\boldsymbol{\xi};q) ,
\end{equation*}
which completes the proof of the proposition for $\texttt{c}=\texttt{a}$.

\subsection{Proof of Proposition \ref{roots:prp} for $\texttt{c}=\texttt{b}$} The above proof for $\texttt{c}=\texttt{a}$ is readily adapted to the case $\texttt{c}=\texttt{b}$. Specifically, after temporarily extending the definition
of $P_{\texttt{b};\mu}  (\boldsymbol{\xi};q,q_0)$ \eqref{HLp:b}, \eqref{Cp:b} to any
$\mu\in \mathbb{Z}^n$ we first verify the corresponding  straightening rule (cf. 
\cite[Proposition 2.1]{nel-ram:kostka}).

\begin{lemma}[Straightening Rule]\label{s-rule-b:lem}    For any $\mu\in\mathbb{Z}^n$, one has that
\begin{equation*}
\mu_j-\mu_{j+1}=-1  \Longrightarrow  P_{\texttt{b};\mu}  (\boldsymbol{\xi};q,q_0) =q  P_{\texttt{b};\mu  +e_{j}-e_{j+1} }  (\boldsymbol{\xi};q,q_0) \quad (j\in \{ 1,\ldots , n-1\})  .
\end{equation*}
\end{lemma}

\begin{proof}
With the notation as in the proof of Lemma \ref{s-rule-a:lem}, we again have that
\begin{equation*}
\frac{C_{\texttt{b}}(\boldsymbol{\xi};q,q_0)}{C_{\texttt{b}}(r_j\boldsymbol{\xi};q,q_0)} =\frac{q-e^{i(\xi_j-\xi_{j+1})}}{1-qe^{i(\xi_j-\xi_{j+1})}} \qquad (j\in \{ 1,\ldots, n-1\} ).
\end{equation*}
The straightening rule thus follows in the same manner as before, except that now $\mu\in\mathbb{Z}^n$ (with $\mu_j-\mu_{j+1}=-1$) and we symmetrize instead with respect to the signed permutation
action (on $\boldsymbol{\xi}$).
\end{proof}

The additional algebraic relations between the variables $\xi_1,\ldots ,\xi_n$, which are  satisfied at  the nodes $\boldsymbol{\xi}=\boldsymbol{\xi}^{(m,n)}_{\texttt{b};\lambda}$, $\lambda\in\Lambda^{(m,n)}_{\texttt{b}}$,
are similarly deduced by exponentiating Eq. \eqref{CEQ:b}:
\begin{align}\label{BAE:b}
e^{2i(m+1)\xi_j} =&
 \left( \frac{1-q_0 e^{i\xi_j}}{e^{i\xi_j}-q_0}\right)  \left( \frac{1-q_1 e^{i\xi_j}}{e^{i\xi_j}-q_1}\right) \\
& \times
\prod_{\substack{ 1\leq k\leq n \\ k\neq j}}  \left( \frac{1-q e^{i(\xi_j+\xi_k)}}{e^{i(\xi_j+\xi_k)}-q}\right) \left( \frac{1-q e^{i(\xi_j-\xi_k)}}{e^{i(\xi_j-\xi_k)}-q}\right) \qquad (j=1,\ldots ,n).
\nonumber
\end{align}
We now arrive at the following affine straightening rule.

\begin{lemma}[Affine Straightening Rule]\label{as-rule-b:lem}    For  any $\mu\in\mathbb{Z}^n$ and variables $\boldsymbol{\xi}$ satisfying Eq. \eqref{BAE:b}, one has that
\begin{equation*}
\mu_1=m+1  \Longrightarrow  P_{\texttt{b};\mu}  (\boldsymbol{\xi};q,q_0) =q_1  P_{\texttt{b};\mu  -e_{1}}  (\boldsymbol{\xi};q,q_0) .
\end{equation*}
\end{lemma}

\begin{proof}
Let $r_{\texttt{b}}$ act on the components of $\boldsymbol{\xi}=(\xi_1,\ldots ,\xi_n)$ by flipping the sign of $\xi_1$:
\begin{equation*}
r_{\texttt{b}}(\xi_1,\ldots ,\xi_n):= (-\xi_1,\xi_2,\ldots ,\xi_{n}).
\end{equation*}
Then Eq. \eqref{BAE:b} with $j=1$ can be rewritten as
\begin{equation*}
e^{2i(m+1)\xi_1}= \frac{C_{\texttt{b}}(r_{\texttt{b}}\boldsymbol{\xi};q,q_0)}{C_{\texttt{b}}(\boldsymbol{\xi};q,q_0)} \left( \frac{1-q_1 e^{i\xi_1}}{q_1e^{-i\xi_1}-1} \right) ,
\end{equation*}
which implies that
\begin{eqnarray*}
\lefteqn{C_{\texttt{b}}(\boldsymbol{\xi};q,q_0) + C_{\texttt{b}}(r_{\texttt{b}}\boldsymbol{\xi};q,q_0) e^{-2i(m+1)\xi_1} =}&& \\
&& q_1 C_{\texttt{b}}(\boldsymbol{\xi};q,q_0) e^{-i\xi_1}+ q_1 C_{\texttt{b}}(r_{\texttt{b}}\boldsymbol{\xi};q,q_0)e^{-i(2m+1)\xi_1} .
\end{eqnarray*}
Multiplication of both sides of the latter identity by $\exp ( i\xi_1\mu_1+\cdots+i\xi_n\mu_n)$, with  $\mu\in \mathbb{Z}^n$ such that
$\mu_1=m+1$ (so  $r_{\texttt{b}}\mu=\mu -2(m+1)e_1$), entails the asserted affine straightening rule when symmetrizing with respect to the signed permutation action (on $\boldsymbol{\xi}$).
\end{proof}

Upon   applying  first the affine straightening rule of Lemma \ref{as-rule-b:lem} to $P_{\texttt{b};\mu}  (\boldsymbol{\xi};q,q_0) $ with
$\mu\in\Lambda^{(m+1,n)}_{\texttt{b}}\setminus\Lambda^{(m,n)}_{\texttt{b}}$, and then rearranging the components of $\mu-e_1$ with the aid of Lemma \ref{s-rule-b:lem} in weakly decreasing order, one infers that for $\boldsymbol{\xi}=\boldsymbol{\xi}^{(m,n)}_{\texttt{b};\lambda}$, $\lambda\in\Lambda^{(m,n)}_{\texttt{b}}$:
\begin{equation*}
P_{\texttt{b};\mu}  (\boldsymbol{\xi};q,q_0) =  q^{\text{m}_{\mu_1}(\mu) -1}  q_1 P_{\texttt{b};\mu  -e_{\text{m}_{\mu_1}(\mu)}}  (\boldsymbol{\xi};q,q_0) .
\end{equation*}
By iterating the process in question we get
\begin{equation*}
P_{\texttt{b};\mu}  (\boldsymbol{\xi};q,q_0) 
= q^{\frac{1}{2}\text{m}_{\mu_1}(\mu) ( \text{m}_{\mu_1}(\mu)-1)}  q_1^{\text{m}_{\mu_1}(\mu)} P_{\texttt{b};\mu-\omega_{\texttt{b};\mu}}  (\boldsymbol{\xi};q,q_0) ,
\end{equation*}
which completes the proof of the proposition for $\texttt{c}=\texttt{b}$.

\section{Specialization to planar domains: determinantal formula for the Christoffel weights}\label{sec5}
In order to convert the cubature rules of Corollaries \ref{hwc-a:cor} and \ref{hwc-b:cor} into effective numerical tools, compact expressions are desired for the Christoffel weights 
$\Delta^{(m,n)}_{\texttt{c};\lambda}$. For $n_{\texttt{c}}=1$ explicit formulas achieving this purpose  can be read-off upon specializing
\cite[Theorem 5]{die-ems:quadrature} (cf. 
Remarks \ref{non-minimal:rem} and \ref{n=1-b:rem}). In this section we generalize the corresponding formulas  for the Christoffel weights
to the planar situation: $n_{\texttt{c}}=2$.
The
cubature formulas of interest are designed to integrate (trigonometric) rational functions over the equilateral triangle and the isosceles  right triangle, respectively; as such they fit within a rich tradition of cubature rules on triangular domains in polynomial  spaces (cf. e.g. Refs. \cite{cow:gaussian,lyn-coo:survey,pap:new}) and trigonometric polynomial spaces
(cf. e.g. Refs. \cite{mun:group,li-sun-xu:discrete-08,li-sun-xu:discrete}), respectively.

\subsection{Integration on the equilateral triangle}
For $n=2$ and $n=3$ the fundamental domain $\mathbb{A}^{(n)}_{\texttt{a}}$ \eqref{A:a} consists of a line segment and an equilateral triangle, respectively.
In these situations, the following proposition provides
a determinantal formula for the Christoffel weights in Corollary \ref{hwc-a:cor}.

\begin{proposition}[Determinantal formula for $\Delta^{(m,n)}_{\texttt{a};\lambda }$, $n_{\texttt{a}}\leq 2$]\label{cw-a:prp}
For $n_{\texttt{a}}=n-1\leq 2$ the Christoffel weights in Corollary \ref{hwc-a:cor} are of the form
\begin{equation}\label{det:a}
\Delta^{(m,n)}_{\texttt{a};\lambda }=
\frac{m}{n}  \left( \det \left[  H^{(m,n)}_{\texttt{a}; j,k} (  \boldsymbol{\xi}^{(m,n)}_{\texttt{a};\lambda} ) \right]_{1\leq j,k\leq n} \right)^{-1}  ,
\end{equation}
with $ H^{(m,n)}_{\texttt{a}; j,k}(\boldsymbol{\xi})$ taken from Eq. \eqref{Hesse:a}.
\end{proposition}

\begin{proof}
The asserted determinantal formula is immediate from the expressions of the Christoffel weights in Theorem \ref{hwc-a:thm} and Corollary \ref{hwc-a:cor} upon invoking the determinantal evaluation formula in  \cite[Proposition 3]{die:finite-dimensional}.
\end{proof}

\begin{remark}
In the planar situation, the cubature of Remark \ref{non-minimal:rem} in the coordinates $X_1=\cos (\xi_1)+\cos(\xi_2)+\cos(\xi_1+\xi_2)$, $X_2=\sin(\xi_1)+\sin(\xi_2)-\sin(\xi_1+\xi_2)$ (with $(\xi_1,\xi_2)\in\mathbb{R}^2$ such that $\xi_1-\xi_2>0$, $\xi_1+2\xi_2>0$ and $2\xi_1+\xi_2<2\pi$) becomes a rule for the integration over the interior region bounded by Steiner's deltoid of area $2\pi$ (cf. e.g. \cite[Section 3]{koo:orthogonal3})
\begin{align}\label{A2:algebraic}
\frac{1}{2\pi^2}\int_{\text{A}^{(3)}_{\texttt{a}}}    f(X_1,X_2) &  \frac{\sqrt{\rho_{\texttt{a}}(X_1,X_2)}}{O_{\texttt{a}}(X_1,X_2;q)}     \text{d}X_1\text{d}X_2\\
&= \sum_{\substack{ l_1, l_2\geq 0\\l_1+l_2\leq m}}
f\bigl(\boldsymbol{X}^{(m,3)}_{\texttt{a};l_1\omega_1+l_2\omega_2}\bigr) \hat{\Delta}^{(m,3)}_{\texttt{a};l_1\omega_1+l_2\omega_2} , \nonumber
\end{align}
 where 
 \begin{align*}
 \text{A}^{(3)}_{\texttt{a}}= &  \{  (X_1,X_2) \in\mathbb{R}^2 \mid  \rho_{\texttt{a}}(X_1,X_2)> 0 \} ,\\
 \rho_{\texttt{a}}(X_1,X_2)=&   8(X_1^3-3X_1X_2^2)-(X_1^2+X_2^2+9)^2+108 ,\\
 O_{\texttt{a}}(X_1,X_2;q)=&1+q^6-(q+q^5)\bigl(X_1^2+X_2^2-3\bigr) \\
 & +(q^2+q^4)\bigl(6-5(X_1^2+X_2^2) +2(X_1^3-3X_1X_2^2)\bigr) \\
 &-q^3\bigl( (X_1^2+X_2^2+3)^2-4(X_1^3-3X_1X_2^2)-16  \bigr) ,\\
  \hat{\Delta}^{(m,3)}_{\texttt{a};l_1\omega_1+l_2\omega_2} =& \frac{ \rho_{\texttt{a}}  \bigl(\boldsymbol{X}^{(m,3)}_{\texttt{a};l_1\omega_1+l_2\omega_2}\bigr)  }{O_{\texttt{a}} \bigl(\boldsymbol{X}^{(m,2)}_{\texttt{b};l_1\omega_1+l_2\omega_2};q\bigr)}   {\Delta}^{(m,3)}_{\texttt{a};l_1\omega_1+l_2\omega_2} ,
 \end{align*}
 and $f(X_1,X_2)$ is allowed to be any polynomial of total degree at most $2m-1$ in $X_1,X_2$.  For $q=0$ the cubature rule \eqref{A2:algebraic} can be found in
\cite[Section 5.2]{li-sun-xu:discrete-08} (cf. also Section \ref{SchurA:sec} below) and for $q\to 1$ in  \cite[Section 3.4]{mun:group},  \cite[Section 5.3]{li-sun-xu:discrete-08} and \cite[Section 4.1]{hri-mot-pat:cubature} (cf. also Section \ref{monomialA:sec} below).
\end{remark}

\subsection{Integration on the isosceles right triangle}
The fundamental domain $\mathbb{A}^{(n)}_{\texttt{b}}$ \eqref{A:b} boils down  to  a line segment and an isosceles right triangle, respectively, when $n=1$ and $n=2$.
The corresponding 
Christoffel weights in Corollary \ref{hwc-b:cor} are then given by the following determinantal formula.

\begin{proposition}[Determinantal formula for $\Delta^{(m,n)}_{\texttt{b};\lambda }$, $n_{\texttt{b}}\leq 2$]\label{cw-b:prp}
For $n_{\texttt{b}}=n\leq 2$ the Christoffel weights in Corollary \ref{hwc-b:cor} are of the form
\begin{equation}\label{det:b}
\Delta^{(m,n)}_{\texttt{b};\lambda }=
  \left( \det \left[  H^{(m,n)}_{\texttt{b}; j,k} (  \boldsymbol{\xi}^{(m,n)}_{\texttt{b};\lambda} ) \right]_{1\leq j,k\leq n} \right)^{-1} ,
\end{equation}
with $ H^{(m,n)}_{\texttt{b}; j,k}(\boldsymbol{\xi})$ taken from Eq. \eqref{Hesse:b}.
\end{proposition}

\begin{proof}
The idea of the proof is to provide a corresponding determinantal evaluation formula for the representations of the Christoffel weights  in Theorem \ref{hwc-b:thm} and Corollary \ref{hwc-b:cor}.
To this end one uses that at the cubature nodes $ \boldsymbol{\xi}^{(m,n)}_{\texttt{b};\lambda}$, $\lambda\in\Lambda^{(m,n)}$
the relations in Eq. \eqref{BAE:b} are satisfied.
Specifically, from the explicit formula for the hyperoctahedral Hall-Littlewood polynomial in Eqs.  \eqref{HLp:b}, \eqref{Cp:b} it follows that
\begin{subequations}
 \begin{align}
 &  | C_{\texttt{b}}(\boldsymbol{\xi} ;q,q_0) |^{-2} \sum_{\mu\in \Lambda^{(m,n)}_{\texttt{b}}}  
  \left| P_{\texttt{b};\mu} \bigl(\boldsymbol{\xi} ;q,q_0\bigr) \right|^2  \delta^{(m,n)}_{\texttt{b};\mu } (q)  \\
&  =\sum_{\substack{\sigma, \sigma^\prime\in S_n \\ \epsilon,\epsilon^\prime\in \{ 1,-1\}^n}}   \frac{C_{\texttt{b}}\left( \epsilon_1 \xi_{\sigma_1},\ldots ,\epsilon_n\xi_{\sigma_n} ;q,q_0\right)}{C_{\texttt{b}}\left(\epsilon_1^\prime \xi_{\sigma_1^\prime},\ldots ,\epsilon_n^\prime\xi_{\sigma_n^\prime}  ;q,q_0\right)} 
  G_{\texttt{b}} \left(      \epsilon_1 \xi_{\sigma_1}- \epsilon_1^\prime \xi_{\sigma_1^\prime}   ,\ldots ,\epsilon_n\xi_{\sigma_n}   -\epsilon_n^\prime\xi_{\sigma_n^\prime}          \right)   , \nonumber
  \end{align}
where
\begin{equation}\label{Gb}
 G_{\texttt{b}} (\xi_1,\ldots ,\xi_n)   :=
\sum_{\mu\in \Lambda^{(m,n)}_{\texttt{b}}}  \delta^{(m,n)}_{\texttt{b};\mu } (q)   e^{i\mu_1\xi_1+\cdots +i\mu_n\xi_n}  .
\end{equation}
\end{subequations}
For $n=1$ this simplifies to
\begin{subequations}
\begin{equation}
2(m+1)   +   \frac{C_{\texttt{b}}(\xi_1;q,q_0)}{C_{\texttt{b}}(-\xi_1;q,q_0)}G_{\texttt{b}}(2\xi_1) + \frac{C_{\texttt{b}}(-\xi_1;q,q_0)}{C_{\texttt{b}}(\xi_1;q,q_0)} G(-2\xi_1) 
\end{equation}
with
\begin{equation}
 G_{\texttt{b}} (\xi_1)=\sum_{0\leq\mu_1\leq m} e^{i\mu_1\xi_1} ,
\end{equation}
\end{subequations}
whereas for $n=2$ one arrives at
\begin{subequations}
\begin{equation}
\sum_{\substack{\sigma, \sigma^\prime\in S_2 \\ \epsilon,\epsilon^\prime\in \{ 1,-1\}^2}}   \frac{C_{\texttt{b}}\left( \epsilon_1 \xi_{\sigma_1},\epsilon_2\xi_{\sigma_2} ;q,q_0\right)}{C_{\texttt{b}}\left(\epsilon_1^\prime \xi_{\sigma_1^\prime} ,\epsilon_2^\prime\xi_{\sigma_2^\prime}  ;q,q_0\right)} 
  G_{\texttt{b}} \left(      \epsilon_1 \xi_{\sigma_1}- \epsilon_1^\prime \xi_{\sigma_1^\prime}   ,\epsilon_2\xi_{\sigma_2}   -\epsilon_2^\prime\xi_{\sigma_2^\prime}          \right) 
\end{equation}
with
\begin{equation}
 G_{\texttt{b}} (\xi_1,\xi_2)=  \sum_{m\geq \mu_1> \mu_2\geq 0}    e^{i\mu_1\xi_1+i\mu_2\xi_2}    +\frac{1}{1+q}  \sum_{m\geq \mu_1= \mu_2\geq 0}    e^{i\mu_1\xi_1+i\mu_2\xi_2}  .
\end{equation}
\end{subequations}
In both situations, summation of the pertinent geometric series and subsequent elimination of all instances of 
$e^{\pm im\xi_j}$ ($j=1,\ldots ,n$) by means of the relations in Eq. \eqref{BAE:b}, gives rise to a (cumbersome) expression that can be rewritten as
 $ \det \left[  H^{(m,n)}_{\texttt{b}; j,k} (  \boldsymbol{\xi} ) \right]_{1\leq j,k\leq n} $.
 \end{proof}

 \begin{remark}
In the planar situation, the cubature of Remark \ref{n=1-b:rem} in the coordinates $X_1=2\cos (\xi_1)+2\cos(\xi_2)$, $X_2=2\cos(\xi_1+\xi_2)+2\cos(\xi_1-\xi_2)$ (with $(\xi_1,\xi_2)\in\mathbb{R}^2$ such that $\pi>\xi_1>\xi_2>0$) becomes
a rule for the integration over the region bounded by the parabola $X_1^2-4X_2= 0$ and the lines $ -2X_1+X_2+4=0$ and $ 2X_1+X_2+4=0$ (cf. e.g. \cite[Section 3]{koo:orthogonal1}):
\begin{align}\label{BC2:algebraic}
\frac{1}{4\pi^2}\int_{\text{A}^{(2)}_{\texttt{b}}}    f(X_1,X_2) &  \frac{\sqrt{\rho_{\texttt{b}}(X_1,X_2)}}{O_{\texttt{b}}(X_1,X_2;q,q_0)}     \text{d}X_1\text{d}X_2\\
&= \sum_{m\geq \lambda_1\geq \lambda_2\geq 0}
f\bigl(\boldsymbol{X}^{(m,2)}_{\texttt{b};(\lambda_1,\lambda_2)}\bigr) \hat{\Delta}^{(m,2)}_{\texttt{b};(\lambda_1,\lambda_2)} , \nonumber
\end{align}
 where 
 \begin{align*}
 \text{A}^{(2)}_{\texttt{b}}= &  \{  (X_1,X_2) \in\mathbb{R}^2 \mid  X_1^2-4X_2> 0,\,  -2|X_1|+X_2+4 >0  \} ,\\
 \rho_{\texttt{b}}(X_1,X_2)=&    (X_1^2-4X_2)( 2X_1+X_2+4)( -2X_1+X_2+4) ,\\
 O_{\texttt{b}}(X_1,X_2;q,q_0)=&\bigl(1+q^4-(q+q^3)X_2+q^2(X_1^2-2X_2-2)\bigr) \\
 &\times \bigl(1+q_0^4-(q_0+q_0^3)X_1+q_0^2(X_2+2)\bigr),\\ 
  \hat{\Delta}^{(m,2)}_{\texttt{b};(\lambda_1,\lambda_2)} =& \frac{ \rho_{\texttt{b}}  \bigl(\boldsymbol{X}^{(m,2)}_{\texttt{b};(\lambda_1,\lambda_2)}\bigr)  }{O_{\texttt{b}} \bigl(\boldsymbol{X}^{(m,2)}_{\texttt{b};(\lambda_1,\lambda_2)};q,q_0\bigr)}   {\Delta}^{(m,2)}_{\texttt{b};(\lambda_1,\lambda_2)} ,
 \end{align*}
 and $f(X_1,X_2)$ is allowed to be any polynomial of total degree at most $2m$ in $X_1,X_2$.
 For $q,q_0\in \{ 0,1\}$ the cubature rule \eqref{BC2:algebraic} falls within class of planar cubatures studied in greater generality  in \cite{moo-pat:cubature,xu:minimal,moo-mot-pat:gaussian} and \cite[Section 4.2]{hri-mot-pat:cubature} (cf. also Sections \ref{SchurB:sec} and  \ref{SymmetrizedB:sec}  below).
  \end{remark}

\subsection{Numerical test of the determinantal formula for  $\Delta^{(m,n)}_{\texttt{c};\lambda }$ with $n_{\texttt{c}}>2$}
It is expected that the determinantal formulas for the Christoffel weights in Propositions \ref{cw-a:prp} and \ref{cw-b:prp} in fact persist for $n_{\texttt{c}}>2$, but a direct confirmation along the lines of the above proofs for $n_{\texttt{c}}=2$ would quickly turn into a very tedious computational tour de force. On the other hand, for $f(\boldsymbol{\xi})\equiv 1$ we can evaluate the multivariate integral under consideration explicitly in closed form by means of
the orthogonality relations in Eqs. \eqref{or-d:a} and \eqref{or-d:b} (with $\mu=\nu=0$), in combination with Macdonald's constant term identity (cf. \cite[\S 10]{mac:orthogonal})
\begin{equation}
P_{\texttt{a};0}  (\boldsymbol{\xi};q) =P_{\texttt{b}; 0}  (\boldsymbol{\xi};q,q_0) = 
 \prod_{1\leq j<k\leq n}  \frac{1-q^{1+k-j}}{1-q^{k-j}} =   \prod_{1\leq j\leq n}  \frac{1-q^j}{1-q} .
\end{equation}
By comparing with the (exact)  value of the integral produced by the cubature rule, this entails
the following identity for the Christoffel weights:
\begin{equation}\label{CW-rel}
\sum_{\lambda\in\Lambda^{(m,n)}_{\texttt{c}}}  | C_{\texttt{c}}(\boldsymbol{\xi}^{(m,n)}_{\texttt{c};\lambda}) |^{-2}    \Delta^{(m,n)}_{\texttt{c};\lambda }
=  \prod_{1\leq j\leq n}  \frac{1-q}{1-q^j} 
\end{equation}
(where $C_{\texttt{c}}(\boldsymbol{\xi}):=C_{\texttt{a}}(\boldsymbol{\xi};q) $ \eqref{Cp:a} if $\texttt{c}=\texttt{a}$ and $C_{\texttt{c}}(\boldsymbol{\xi}):=C_{\texttt{b}}(\boldsymbol{\xi};q,q_0) $ \eqref{Cp:b}  if $\texttt{c}=\texttt{b}$).
For $m=1$, the identity under consideration specializes to
\begin{equation}\label{CW-rel-m=1}
\sum_{0\leq j\leq n_{\texttt{c}}}     | C_{\texttt{c}}(\boldsymbol{\xi}^{(1,n)}_{\texttt{c};\omega_{\texttt{c}; j}}) |^{-2}    \Delta^{(1,n)}_{\texttt{c};\omega_{\texttt{c};j} }=
  \prod_{1\leq j\leq n}  \frac{1-q}{1-q^j} ,
\end{equation}
where  $\omega_{\texttt{c};j}:=\omega_j$ \eqref{fwb:a} if $\texttt{c}=\texttt{a}$ and $\omega_{\texttt{c};j}:=e_1+\cdots +e_j$ if $\texttt{c}=\texttt{b}$, with the convention that
$\omega_{\texttt{c};0}:=0$.

Tables \ref{table:a} and \ref{table:b} provide  numerical examples  for $m=1$ and $n_{\texttt{c}}=3$ that exhibit the cubature nodes, the values of the Christoffel weights predicted by the determinantal formula, and
the values of the (hyperoctahedral) Hall-Littlewood orthogonality measure at the nodes.
The data in these tables were computed with Maple using a precision of 8  digits. The quadratic convergence of Newton's method for the computation of the nodes in question via Eqs. \eqref{CEQ:a} and \eqref{CEQ:b} is illustrated by Tables \ref{table:aNewton}  and \ref{table:bNewton}, respectively; in order to be able to show the convergence up to the fourth iteration we relied on a high precision computation in Maple of 50 digits.

The data of Tables \ref{table:a} and \ref{table:b}
are compatible with the equality in Eq. \eqref{CW-rel-m=1} (within the numerical precision of the tables). Indeed, when $\texttt{c}=\texttt{a}$  the LHS yields $0.53850$, which coincides
with the value
$\frac{15625}{29016}$ on the RHS in five decimals. Similarly, when $\texttt{c}=\texttt{b}$  the LHS yields $0.67205$,
which differs from the value
$\frac{125}{186}$ on the RHS by a unit in the fifth
decimal (caused by the rounding error stemming from the data of the table).

As a second check we have tested the cubature rules in question with the determinantal expressions for the Christoffel weights
beyond the domain of exact integration.
To this end the function
\begin{subequations}
 \begin{equation}\label{testfunction:a}
{\textstyle \exp(\frac{1}{2}\cos \xi_1+\cdots+\frac{1}{2}\cos\xi_n)/O_{\texttt{a}}(\boldsymbol \xi;q) }
\end{equation}
($=\exp\bigl( \frac{1}{2}\text{Re}(M_{a;\omega_1(\boldsymbol \xi)})\bigr)/O_{\texttt{a}}(\boldsymbol \xi;q)$) with $q=\frac{1}{5}$ was integrated in Maple with a precision
of 8 digits against the weight function $\rho_{\texttt{a}}(\boldsymbol{\xi})$, both for $n=3$ (when the determinantal expression is justified
by Proposition \ref{det:a}) and for $n=4$ (when the determinantal expression is conjectural). Table \ref{table:acubature} shows for $m=1$ that in both cases the corresponding
Hall-Littlewood cubature from Corollary \ref{hwc-a:cor}  performs somewhat better than the corresponding Schur cubature from Proposition \ref{sc-a:prp} (below).
Similarly,  the function
\begin{equation}\label{testfunction:b}
\exp(\cos \xi_1+\cdots+\cos\xi_n)/O_{\texttt{b}}(\boldsymbol \xi;q,q_0)
\end{equation}
\end{subequations}
($=\exp\bigl(\frac{1}{2}  M_{\texttt{b};\omega_1(\boldsymbol \xi)}\bigr)/O_{\texttt{b}}(\boldsymbol \xi;q,q_0)$) with $q=\frac{1}{5}$, $q_0=\frac{1}{3}$ and $q_1=\frac{1}{7}$ was integrated in Maple with a precision of 8 digits  against the weight function $\rho_{\texttt{b}}(\boldsymbol{\xi})$; Table \ref{table:bcubature}  reveals for $m=1$ that  the corresponding hyperoctahedral Hall-Littlewood cubature  from Corollary \ref{hwc-b:cor} significantly outperforms
the symplectic Schur cubature from Eq. \eqref{ss-rule} (below), both when $n=2$ (with the determinantal expression being justified
by Proposition \ref{det:b}) and when $n=3$ (with the determinantal expression being conjectural).

 Table \ref{table:planar-different-m} illustrates that by augmenting the number of nodes there is a clear tendency in both examples for the planar (hyperoctahedral) Hall-Littlewood cubature to perform significantly better than the (symplectic) Schur cubature. To achieve the required accuracy the latter table was computed  in Maple with a precision of 12 digits.

\begin{table}[hbt]
\centering
\caption{Hall-Littlewood cubature on the  tetrahedron $\mathbb{A}^{(4)}_{\texttt{a}}$ \eqref{A:a}: cubature nodes, Christoffel weights (via the determinantal formula \eqref{det:a}), and orthogonality measure for $n_{\texttt{a}}=3$ and $m=1$, with
$q=\frac{1}{5}$.}
\label{table:a}
\begin{tabular}{@{}|c|ccc|c|@{}} 
\toprule
 & $\boldsymbol{\xi}^{(1,4)}_{\texttt{a};\omega_{\texttt{a}; j}}$ & $  \Delta^{(1,4)}_{\texttt{a};\omega_{\texttt{a};j} }$ &  $  | C_{\texttt{a}}(\boldsymbol{\xi}^{(1,4)}_{\texttt{a};\omega_{\texttt{a}; j}}) |^{-2}   $  \\
\midrule
$j=0$ & $(1.7848, 0.58020,-0.58020,-1.7848)$ &$2.6453\cdot 10^{-3}$ & $50.892$ \\ 
$j=1$ &  $(2.9276, 0.21398,-0.99059 ,-2.1510)$ &$2.6453\cdot 10^{-3}$ &  $50.892$ \\ 
$j=2$ & $(2.5614, 1.3568,  -1.3568, -2.5614)$ & $2.6453\cdot 10^{-3}$ &  $50.892$  \\ 
$j=3$& $(2.1510, 0.99059, -0.21398 , -2.9276)$ &$2.6453\cdot 10^{-3}$ &  $50.892$  \\ 
\bottomrule
\end{tabular} 
\end{table}

\begin{table}[hbt]
\centering
\caption{Hyperoctahedral Hall-Littlewood cubature on the  tetrahedron $\mathbb{A}^{(3)}_{\texttt{b}}$ \eqref{A:b}:
cubature nodes, Christoffel weights (via the determinantal formula \eqref{det:b}), and orthogonality measure for $n_{\texttt{b}}=3$ and $m=1$, with $q=\frac{1}{5}$, $q_0=\frac{1}{3}$, $q_1=\frac{1}{7}$.}
\label{table:b}
\begin{tabular}{@{}|c|ccc|c|@{}} 
\toprule
 & $\boldsymbol{\xi}^{(1,3)}_{\texttt{b};\omega_{\texttt{b}; j}}$ & $  \Delta^{(1,3)}_{\texttt{b};\omega_{\texttt{b};j} }$ &  $  | C_{\texttt{b}}(\boldsymbol{\xi}^{(1,3)}_{\texttt{b};\omega_{\texttt{b}; j}}) |^{-2}   $  \\
\midrule
$j=0$ & $(1.6920, 1.1134,0.56095)$ &$9.1533\cdot 10^{-4}$ & $98.915$ \\ 
$j=1$ & (2.3903, 1.1508, 0.57998) &$1.0877\cdot 10^{-3}$  & $232.57$ \\ 
$j=2$ &(2.4257, 1.7964, 0.60785) &$1.1607\cdot 10^{-3}$  & $212.18$\\ 
$j=3$& (2.4470, 1.8327, 1.2423) & $1.1394\cdot 10^{-3}$ &$72.198$  \\ 
\bottomrule
\end{tabular} 
\end{table}

{\small
\begin{table}[hbt] 
\centering
\caption{Euclidean distance between the node $\boldsymbol{\xi}^{(1,4)}_{\texttt{a};\omega_{\texttt{a}; j}}$ and the Newton iterates of Eq. \eqref{CEQ:a} starting from the initial estimate $\frac{2\pi}{5}(\omega_{\texttt{a};j}+\varrho_{\texttt a })$, with $q=\frac{1}{5}$.}
\label{table:aNewton}
\begin{tabular}{@{}|c|ccccc|@{}} 
\toprule
 &  0  & 1 &  2 & 3 &4  \\
\midrule
$j=0$ &  $1.57  \cdot 10^{-1}$ & $8.49\cdot 10^{-4}$ & $9.32 \cdot 10^{-8}$ & $1.08 \cdot 10^{-15}$ & $1.53 \cdot 10^{-31}$  \\ 
$j=1$ &  && idem &&\\
$j=2$ &  && idem &&\\
$j=3$ &  && idem &&\\
\bottomrule
\end{tabular} 
\end{table}
}

{\small
\begin{table}[hbt] 
\centering
%:
\caption{Euclidean distance between the node $\boldsymbol{\xi}^{(1,3)}_{\texttt{b};\omega_{\texttt{b}; j}}$ and the Newton iterates of Eq. \eqref{CEQ:b} starting from the initial estimate $\frac{\pi}{5}(\omega_{\texttt{b};j}+\varrho_{\texttt b })$, with $q=\frac{1}{5}$, $q_0=\frac{1}{3}$, $q_1=\frac{1}{7}$.
}

\label{table:bNewton}
\begin{tabular}{@{}|c|ccccc|@{}} 
\toprule
 &  0  & 1 &  2 & 3 &4  \\
\midrule
 $j=0$ &   $2.50 \cdot 10^{-1}$ &   $3.35 \cdot 10^{-3} $ &   $7.49 \cdot 10^{-7} $ & $ 4.02 \cdot 10^{-14}$ & $1.29 \cdot 10^{-28}$   \\ 

 $j=1$  
     & $1.69  \cdot 10^{-1}$ 
     &  $8.19 \cdot 10^{-4}$ 
     & $4.82 \cdot 10^{-8}$ 
     &   $1.66 \cdot 10^{-16}$ 
     &  $2.11 \cdot 10^{-33}$   \\

$j=2$ 
& $1.26 \cdot 10^{-1}$ 
& $2.70 \cdot 10^{-4}$ 
&  $3.12  \cdot 10^{-9}$ 
& $6.72 \cdot 10^{-19}$ 
& $3.15 \cdot 10^{-38}$ 
\\

$j=3$ 

& $8.56 \cdot 10^{-2}$ 
& $2.03 \cdot 10^{-4}$ 
& $9.34 \cdot 10^{-10}$ 
& $4.86 \cdot 10^{-20}$  
& $1.60 \cdot 10^{-40}$ 
\\

\bottomrule
\end{tabular} 
\end{table}
}

{\small
\begin{table}[hbt] 
\centering
\caption{Comparison for $m=1$ of the Hall-Littlewood cubature HLC (Corollary \ref{hwc-a:cor})  using the determinantal formula \eqref{det:a} and  the Schur cubature SC (Proposition \ref{sc-a:prp}), when integrating
the testfunction \eqref{testfunction:a} with $q=\frac{1}{5}$ against $\rho_{\texttt{a}}(\boldsymbol{\xi})$.}
\label{table:acubature}
\begin{tabular}{@{}|c|cc|cc|@{}} 
\toprule
 &\multicolumn{2}{c|}{$n_{\texttt{a}}=2$}  & \multicolumn{2}{c|}{$n_{\texttt{b}}=3$}  \\
 &  Value  & Relative &  Value & Relative   \\
  &  Integral & Error &  Integral & Error   \\
\midrule
Maple & 0.7317  &  & 0.5825 &    \\ 
HLC & 0.7450 &  $1.8\cdot 10^{-2}$ & 0.5926 & $1.7\cdot 10^{-2}$ \\
SC & 0.6862 & $6.2\cdot 10^{-2}$  & 0.5452 & $6.4 \cdot 10^{-2}$ \\
\bottomrule
\end{tabular} 
\end{table}
}

{\small
\begin{table}[hbt] 
\centering
\caption{Comparison for $m=1$ of the hyperoctahedral Hall-Littlewood cubature HHLC (Corollary \ref{hwc-b:cor})  using the determinantal formula \eqref{det:b} and  the  symplectic Schur cubature SSC (Eq. \eqref{ss-rule}), when integrating
the testfunction \eqref{testfunction:b} with $q=\frac{1}{5}$, $q_0=\frac{1}{3}$ and $q_1=\frac{1}{7}$ against $\rho_{\texttt{b}}(\boldsymbol{\xi})$.}
\label{table:bcubature}
\begin{tabular}{@{}|c|cc|cc|@{}} 
\toprule
 &\multicolumn{2}{c|}{$n_{\texttt{b}}=2$}  & \multicolumn{2}{c|}{$n_{\texttt{b}}=3$}  \\
 &  Value  & Relative &  Value & Relative   \\
  &  Integral & Error &  Integral & Error   \\
\midrule
Maple & 1.17979  &  & 0.964386 &    \\ 
HHLC & 1.18029 & $4.2 \cdot 10^{-4}$ & 0.964801 & $4.3 \cdot 10^{-4}$ \\
SSC & 1.11198 &$5.7 \cdot 10^{-2}$ &  0.905819 & $6.1 \cdot 10^{-2}$\\
\bottomrule
\end{tabular} 
\end{table}
}

{\small
\begin{table}[hbt] 
\centering
\caption{ Comparison of the relative cubature errors when integrating the testfunction \eqref{testfunction:a}, \eqref{testfunction:b} against $\rho_{\texttt{c}}(\boldsymbol{\xi})$ for $n_{\texttt{c}}=2$
($q=\frac{1}{5}$, $q_0=\frac{1}{3}$ and $q_1=\frac{1}{7}$).}
\label{table:planar-different-m}
\begin{tabular}{@{}|c|cccc|@{}} 
\toprule
 &  $m=1$  & $m=2$ &  $m=3$ & $m=4$    \\
\midrule
HLC  & $ 1.8 \cdot 10^{-2}$ & $3.2\cdot 10^{-4}$ & $2.4 \cdot 10^{-6}$ & $9.8 \cdot 10^{-9}$  \\ 
SC & $6.2 \cdot 10^{-2}$ & $1.3\cdot 10^{-2}$ & $2.5\cdot 10^{-3}$ & $5.4 \cdot 10^{-4}$  \\  
\hline
HHLC  & $ 4.2 \cdot 10^{-4}$ & $1.8\cdot 10^{-5}$ & $1.4 \cdot 10^{-7}$ & $5.7 \cdot 10^{-10}$  \\ 
SSC & $ 5.7 \cdot 10^{-2}$ & $6.7\cdot 10^{-3}$ & $7.5\cdot 10^{-4}$ & $8.3 \cdot 10^{-5}$  \\  
\bottomrule
\end{tabular} 
\end{table}
}

\section{Degenerations: $q=0$ and $q= 1$}\label{sec6}
The (hyperoctahedral) Hall-Littlewood cubatures of Theorems \ref{hwc-a:thm} and \ref{hwc-b:thm} turn out to unify several previous rules from the literature. In this section we identify a few examples
stemming from the specializations $q= 0$ and $q= 1$.

\subsection{Schur cubature for $\texttt{c}=\texttt{a}$}\label{SchurA:sec}
At $q=0$ the Hall-Littlewood polynomial $P_{\texttt{a};\mu} (\boldsymbol{\xi} ;q)$ \eqref{HLp:a}, \eqref{Cp:a} simplifies to a Schur polynomial. Theorem \ref{hwc-a:thm} (in its its algebraic reformulation of Remark \ref{non-minimal:rem}) then reduces to a more elementary cubature rule from Refs. \cite{li-xu:discrete,moo-pat:cubature}. In the present formulation this rule is well-suited to integrate homogeneous symmetric polynomials against the density of the circular unitary ensemble.

\begin{proposition}[Schur Cubature: $\texttt{c}=\texttt{a}$]\label{sc-a:prp}
For $q=0$ (and $m\in\mathbb{Z}_{>0}$), the cubature rule in Theorem \ref{hwc-a:thm} specializes to
\begin{align}\label{sc:a}
\frac{1}{(2\pi )^{n-1} n^{1/2}} &  \int_{\mathbb{A}^{(n)}_{\texttt{a}}}   f(\boldsymbol{\xi}) \rho_{\texttt{a}} (\boldsymbol{\xi})  \text{d} \boldsymbol{\xi} = \\
&
\frac{1}{n (n+m)^{n-1}}\sum_{\lambda\in\Lambda^{(m,n)}_{\texttt{a}}}      f \Bigl(\frac{2\pi (\varrho_{\texttt{a}}+\lambda)}{m+n}
 \Bigr)   \rho_{\texttt{a}} \Bigl(\frac{2\pi (\varrho_{\texttt{a}}+\lambda)}{m+n}\Bigr)  , \nonumber
\end{align}
where $\varrho_{\texttt{a}}=( \varrho_{\texttt{a};1},\ldots ,\varrho_{\texttt{a};n})$ and $f(\boldsymbol{\xi})$ denotes an arbitrary symmetric polynomial in $\mathbb{P}^{(2m+1,n)}_{\texttt{a}}$. 
\end{proposition}

\begin{proof}
As argued in Remark \ref{bounds-a:rem}, when $q=0$ the nodes are positioned at:
$
 \boldsymbol{\xi}^{(m,n)}_{\texttt{a};\lambda}=\frac{2\pi (\varrho_{\texttt{a}}+\lambda)}{m+n}$ ($\lambda\in\Lambda^{(m,n)}_{\texttt{a}}$).
The corresponding Christoffel weights simplify in this situation to
\begin{align*}
\Delta^{(m,n)}_{\texttt{a};\lambda }= & \left| C_{\texttt{a}}\left(\frac{2\pi (\varrho_{\texttt{a}}+\lambda)}{m+n} ;0\right) \right|^2
\left(\sum_{\mu\in \Lambda^{(m,n)}_{\texttt{a}}}   \left| P_{\texttt{a};\mu} \left(\frac{2\pi (\varrho_{\texttt{a}}+\lambda)}{m+n} ;0 \right) \right|^2    \right)^{-1} \\
= &  \left(\sum_{\mu\in \Lambda^{(m,n)}_{\texttt{a}}}   \left| \det \left[     \exp\left({\frac{2\pi i(\varrho_{\texttt{a};j}+\mu_j)(\varrho_{\texttt{a};k}+\lambda_k)}{m+n}   } \right)   \right]_{1\leq j,k\leq n} \right|^2    \right)^{-1} 
\\
=& \frac{1}{n(n+m)^{n-1}}
\end{align*}
(where the last step relies on well-known discrete orthogonality relations for the antisymmetric monomials, cf. e.g. \cite[\S 13.8]{kac:infnite-dimensional}, \cite[Section 4.2]{die:finite-dimensional}, and \cite[Section 7.4]{moo-pat:cubature}).
It remains to infer that at $q=0$  the cubature formula extends from $f\in \mathbb{P}^{(2m-1,n)}_{\texttt{a}}$ to $f\in \mathbb{P}^{(2m+1,n)}_{\texttt{a}}$, which is done
by carefully reviewing/adapting  the proof of Theorem \ref{hwc-a:thm}.  Indeed, if $\mu\in \Lambda_{\texttt{a}}^{(m+1,n)}\setminus \Lambda_{\texttt{a}}^{(m,n)}$ then
$P_{\texttt{a};\mu}\left(\frac{2\pi (\varrho_{\texttt{a}}+\lambda)}{m+n}; 0\right)=0$  for all
$\lambda\in\Lambda_{\texttt{a}}^{(m,n)}$  (by Proposition \ref{roots:prp}).  So  at $q=0$ the equality between  the orthogonality relations in Eqs.  \eqref{or-c:a} and \eqref{or-d:a} (and thus the cubature rule with $f(\boldsymbol{\xi})$ of the form in Eq. \eqref{kostka}) is in fact valid for any
$\mu\in\Lambda_{\texttt{a}}^{(m+1,n)}$ and $\nu\in\Lambda_{\texttt{a}}^{(m,n)}$ (cf. Remark \ref{jump:rem}).
\end{proof}

Up to rescaling (of the underlying root-- and weight lattices) by the (index) factor $n$, the
cubature rule in Proposition \ref{sc-a:prp} boils down to that of \cite[Theorem 5.8]{li-xu:discrete}. Moreover, the cubature in question can also be seen as a special case of 
\cite[Theorem 7.2]{moo-pat:cubature} corresponding to the root system $R=A_{n-1}$ (cf. also \cite{mun:group}).

\begin{remark}\label{gauss-a:rem}
Proposition \ref{sc-a:prp}  elucidates  in particular that at $q=0$ the  degree of exactness  jumps to the optimal Gaussian value $2m+1$. Indeed, as emphasized in the above proof:
Proposition \ref{roots:prp}  recovers the known fact  that  the $q=0$ cubature nodes
$\frac{2\pi (\varrho_{\texttt{a}}+\lambda)}{m+n}$, $\lambda\in\Lambda_{\texttt{a}}^{(m,n)}$
consist of common roots of the Schur polynomials  $P_{\texttt{a};\mu}(\boldsymbol{\xi}; 0)$, $\mu\in \Lambda_{\texttt{a}}^{(m+1,n)}\setminus \Lambda_{\texttt{a}}^{(m,n)}$ (cf.  \cite[Theorem 5.7]{li-xu:discrete} and  \cite[Section 5]{moo-pat:cubature}).
\end{remark}

\begin{remark}
Proposition \ref{sc-a:prp} confirms that at $q=0$ the
determinantal formula for the Christoffel weights in Proposition \ref{cw-a:prp} persists for arbitrary $n_{\texttt{a}}=n-1\geq 1$. Indeed, for this special parameter value:
\begin{align*}
\det \left[  H^{(m,n)}_{\texttt{a}; j,k} (  \boldsymbol{\xi} )  \right]_{1\leq j,k\leq n} &=
\det \bigl[  (m+n)\delta_{j,k}-1 \bigr]_{1\leq j,k\leq n}  \\
&= {m (m+n)^{n-1}}
\end{align*}
(where $\delta_{j,k}$ refers to the Kronecker delta).

\end{remark}

\subsection{Schur cubature for $\texttt{c}=\texttt{b}$}\label{SchurB:sec}
At $q=0$ the cubature rule in Section \ref{cub:b:sec} becomes of a type studied in Ref.  \cite{die-ems:cubature}. The rules in question are designed to integrate symmetric functions, with prescribed poles at coordinate hyperplanes, against the density of the circular quaternion ensemble.

\begin{proposition}[Schur Cubature: $\texttt{c}=\texttt{b}$]\label{sc-b:prp}
For  $q=0$ (with $q_0,q_1\in  (-1,1)$ and $m\in\mathbb{Z}_{> 0}$), the cubature rule in Corollary \ref{hwc-b:cor} specializes to
\begin{subequations}
\begin{equation}\label{sc:b}
\frac{1}{(2\pi )^{n} }   \int_{\mathbb{A}^{(n)}_{\texttt{b}}}   R_{\texttt{b}}(\boldsymbol{\xi}) \rho_{\texttt{b}} (\boldsymbol{\xi})  \text{d} \boldsymbol{\xi}
=
\sum_{\lambda\in\Lambda^{(m,n)}_{\texttt{b}}}      R_{\texttt{b}}\bigl( \boldsymbol{\xi}^{(m,n)}_{\texttt{b};\lambda}  \bigr)   \rho_{\texttt{b}} \bigl( \boldsymbol{\xi}^{(m,n)}_{\texttt{b};\lambda}\bigr) \Delta^{(m,n)}_{\texttt{b};\lambda } ,
\end{equation}
where
\begin{equation}
R_{\texttt{b}}(\boldsymbol{\xi})=   \frac{f(\boldsymbol{\xi})}{ \prod_{1\leq j\leq n}  (1-2q_0\cos(\xi_j) +q_0^2)}    ,
\end{equation}
with $f(\boldsymbol{\xi})$ denoting an arbitrary symmetric polynomial in $\mathbb{P}^{(2m,n)}_{\texttt{b}}$.
The corresponding cubature roots and  Christoffel weights then take the form
\begin{equation}\label{q=0-nodes:b}
\boldsymbol{\xi}^{(m,n)}_{\texttt{b};\lambda}=   \left(  \xi^{(m+n)}_{\lambda_1+n-1},  \xi^{(m+n)}_{\lambda_2+n-2},\ldots ,  \xi^{(m+n)}_{\lambda_{n-1}+1},  \xi^{(m+n)}_{\lambda_n} \right)
\end{equation}
and 
\begin{equation}
\Delta^{(m,n)}_{\texttt{b};\lambda }= \prod_{1\leq j\leq n}  \Delta_{\lambda_j+n-j}^{(m+n)}  ,
\end{equation}
respectively, where
\begin{equation}\label{cw-n=1:b}
 \Delta_{l}^{(m+n)}:=\Delta^{(m+n-1,1)}_{\texttt{b}; l }=  \left( 2(m+n)+ u_{q_0} \bigl(  \xi^{(m+n)}_l \bigr)+ u_{q_1} \bigl(  \xi^{(m+n)}_l \bigr) \right)^{-1}
\end{equation}
and $\xi^{(m+n)}_l:=\xi^{(m+n-1,1)}_{\texttt{b};l}$ denotes the unique real root of the transcendental equation
\begin{equation}\label{n=1-nodes:b}
2(m+n)\xi + v_{q_0}(\xi) +v_{q_1}(\xi)=  2\pi (l+1)
\end{equation}
($0\leq l < m+n$).
\end{subequations}
\end{proposition}

\begin{proof}
It is immediate from Eq. \eqref{CEQ:b} that at $q=0$ the nodes are of the form in Eq. \eqref{q=0-nodes:b} with $\xi^{(m+n)}_l$ solving Eq. \eqref{n=1-nodes:b}. Moreover, we have that
(cf. \cite[Remark 3.7]{die-ems-zur:completeness})
\begin{equation*}
P_\mu (\boldsymbol{\xi};0,q_0 )=   \frac{\det [  p_{n-j+\mu_j}(\xi_k;q_0)]_{1\leq j,k\leq n}}{\prod_{1\leq j< k\leq n} (2\cos (\xi_j)-2\cos(\xi_k))} ,
\end{equation*}
with
\begin{equation}\label{HL-n=1:b}
p_l(\xi;q_0):=c(\xi;q_0) e^{il\xi} +c(-\xi;q_0)  e^{-il\xi}  ,\quad c(\xi;q_0):= \frac{1-q_0e^{-i\xi}}{1-e^{-2i\xi}}  .
\end{equation}
The corresponding Christoffel weights thus take the form
\begin{align*}
\Delta^{(m,n)}_{\texttt{b};\lambda} =&
 | C_{\texttt{b}}(\boldsymbol{\xi}^{(m,n)}_{\texttt{b};\lambda} ;0,q_0) |^2
\Biggl(
\sum_{\mu\in \Lambda^{(m,n)}_{\texttt{b}}}    \left| P_{\texttt{b};\mu} \bigl(\boldsymbol{\xi}^{(m,n)}_{\texttt{b};\lambda} ;0,q_0\bigr)  \right|^2 \Biggl)^{-1}  \\
=&   \left| \prod_{1\leq k\leq n}  c \bigl(    \xi^{(m+n)}_{n-k+\lambda_k}      ;q_0\bigr) \right|^2\\
&\times \Biggl(
\sum_{\mu\in \Lambda^{(m,n)}_{\texttt{b}}}    \left( \det [  p_{n-j+\mu_j}(\xi^{(m+n)}_{n-k+\lambda_k};q_0)]_{1\leq j,k\leq n} \right)^2 \Biggl)^{-1} .
\end{align*}
The sum in the denominator can be rewritten as:
\begin{align*}
&\sum_{\mu\in \Lambda^{(m,n)}_{\texttt{b}}}    \left( \det [  p_{n-j+\mu_j}(\xi^{(m+n)}_{n-k+\lambda_k};q_0)]_{1\leq j,k\leq n} \right)^2\\
&=\sum_{m+n>\nu_1>\nu_2>\dots > \nu_n\geq 0}    \left( \det [  p_{\nu_j}(\xi^{(m+n)}_{n-k+\lambda_k};q_0)]_{1\leq j,k\leq n} \right)^2 \\
&\stackrel{(i)}{=}     \det \left[    \sum_{0\leq \nu <m+n}   p_{\nu}(\xi^{(m+n)}_{n-j+\lambda_j};q_0)    p_{\nu}(\xi^{(m+n)}_{n-k+\lambda_k };q_0) \right]_{1\leq j,k\leq n}  \\
& \stackrel{(ii)}{= } \det \left( \text{diag} \left[    \sum_{0\leq \nu <m+n}   \left| p_{\nu} (\xi^{(m+n)}_{n-k+\lambda_k};q_0) \right|^2   \right]_{1\leq k\leq n}   \right) ,
\end{align*}
where we relied on the Cauchy-Binet formula  $(i)$ and on a special instance of  the  orthogonality in Eq. \eqref{d-ortho:b} corresponding to a single variable on $m+n$ nodes $(ii)$.

The upshot is that the  Christoffel weights factorize at $q=0$ as follows:
\begin{equation*}
\Delta^{(m,n)}_{\texttt{b};\lambda} = \prod_{1\leq j \leq n}  \Delta^{(m+n)}_{n-j+\lambda_j}
\end{equation*}
with
\begin{align*}
 \Delta^{(m+n)}_{l}=&  \left| c \bigl(    \xi^{(m+n)}_{l}      ;q_0\bigr) \right|^2 \left( \sum_{0\leq\nu <m+n}   \left| p_{\nu} (\xi^{(m+n)}_{l};q_0) \right|^2 \right)^{-1} \\
 =&   \Delta_{\texttt{b};l}^{(m+n-1,1)}=   \left( 2(m+n)+ u_{q_0} \bigl(  \xi^{(m+n)}_l \bigr)+ u_{q_1} \bigl(  \xi^{(m+n)}_l \bigr) \right)^{-1} ,
\end{align*}
where the last equality hinges on the formula in Proposition \ref{cw-b:prp} (with $n_{\texttt{b}}=1$ and $m+n$ nodes). 
\end{proof}

The rule in Proposition \ref{sc-b:prp} boils down to a special case of
\cite[Theorem 2]{die-ems:cubature} with $d=\tilde{d}=1$ and $\epsilon_\pm,\tilde{\epsilon}_\pm =1$. It fits within a general framework due to Berens, Schmid and Xu designed to promote Gaussian quadratures to
cubature rules for symmetric functions, cf.
 \cite[Equation (8)]{ber-sch-xu:multivariate}.
If in addition $q_0=q_1=0$, then our rule  simplifies further: 
\begin{align}\label{ss-rule}
\frac{1}{(2\pi )^{n} }   \int_{\mathbb{A}^{(n)}_{\texttt{b}}} & f (\boldsymbol{\xi}) \rho_{\texttt{b}} (\boldsymbol{\xi})  \text{d} \boldsymbol{\xi}
=\\
&\frac{1}{    2^n (m+n+1)^n   }
\sum_{\lambda\in\Lambda^{(m,n)}_{\texttt{b}}}      f \Bigl( \frac{\pi (\varrho_{\texttt{b}}+\lambda)}{m+n+1} \Bigr)   \rho_{\texttt{b}} \Bigl( \frac{\pi (\varrho_{\texttt{b}}+\lambda)}{m+n+1} \Bigr)  ,\nonumber
\end{align}
where $\varrho_{\texttt{b}}:=( \varrho_{\texttt{b};1},\ldots ,\varrho_{\texttt{b};n})$ and $f(\boldsymbol{\xi})$ denotes an arbitrary symmetric polynomial in $\mathbb{P}^{(2m+1,n)}_{\texttt{b}}$.
As before, the jump  to the optimal Gaussian degree of exactness $2m+1$, at vanishing parameter values,  is a consequence of the fact that
the pertinent cubature nodes $ \frac{\pi (\varrho_{\texttt{b}}+\lambda)}{m+n+1}$, $\lambda\in\Lambda^{(m,n)}_{\texttt{b}}$
are common roots of the (symplectic) Schur polynomials  $P_{\texttt{b};\mu}(\boldsymbol{\xi}; 0,0,0)$, $\mu\in \Lambda_{\texttt{b}}^{(m+1,n)}\setminus \Lambda_{\texttt{b}}^{(m,n)}$
(cf. Remark \ref{jump:rem}).
In fact, the (symplectic) Schur cubature rule in Eq. \eqref{ss-rule} can be identified as a special case of the Gaussian cubature rule in \cite[Eqs. (9.2a), (9.2b)]{die-ems:exact} with $\epsilon_\pm =1$. Closely related cubature rules were discussed in \cite[Section 5]{hri-mot:discrete}.

\begin{remark}
Proposition \ref{sc-b:prp} confirms that at $q=0$ the
determinantal formula for the Christoffel weights in Proposition \ref{cw-b:prp} persists for arbitrary $n_{\texttt{b}}=n\geq 1$, since for this special parameter value:
\begin{align*}
  \det \left[  H^{(m,n)}_{\texttt{b}; j,k} (  \boldsymbol{\xi} ) \right]_{1\leq j,k\leq n} &=
 \det \Bigl( \text{diag} \bigl[ 2(m+n) +u_{q_0}(\xi_j)+ u_{q_1}(\xi_j) \bigr]_{1\leq j\leq n} \Bigr)   \\
&= \prod_{1\leq j\leq n} \bigl(2(m+n) +u_{q_0}(\xi_j)+ u_{q_1}(\xi_j) \bigr) .
\end{align*}
\end{remark}

\subsection{Monomial cubature}\label{monomialA:sec}
At $q=1$ the Hall-Littlewood polynomial degenerates to a symmetric monomial $P_{\texttt{a};\mu}(\boldsymbol{\xi};1)=N_{\texttt{a};\mu}  M_{\texttt{a};\mu}  (\boldsymbol{\xi})$, $\mu\in\Lambda_{\texttt{a}}^{(m,n)}$.
The corresponding cubature rule can be found in \cite[Section 5.3]{li-xu:discrete} in the algebraic formulation of Remark \ref{non-minimal:rem}  (upon rescaling  the variables with the index $n$), and in \cite[Section 3.2]{mun:group} and \cite[Section 3.2]{hri-mot-pat:cubature} (upon specialization to the root system $R=A_{n-1}$):
\begin{subequations}
\begin{equation}\label{mc:a}
\frac{1}{(2\pi )^{n-1} n^{1/2}}   \int_{\mathbb{A}^{(n)}_{\texttt{a}}}   f(\boldsymbol{\xi})   \text{d} \boldsymbol{\xi} =
\frac{1}{n \, m^{n-1}}\sum_{\lambda\in\Lambda^{(m,n)}_{\texttt{a}}}      f \left(\frac{2\pi \lambda}{m}
 \right)   \delta^{(m,n)}_{\texttt{a};\lambda}  (1)
\end{equation}
for   $  f(\boldsymbol{\xi}) $ in $\mathbb{P}_{\texttt{a}}^{(2m-1,n)}$, where
\begin{equation}\label{cw-mc:a}
\delta^{(m,n)}_{\texttt{a};\lambda }(1)= \prod_{\substack{1\leq j<k\leq n\\ \lambda_j-\lambda_k=0}}  \frac{k-j}{1+k-j}
\prod_{\substack{1\leq j<k\leq n\\ \lambda_j-\lambda_k=m}}  \frac{n-k+j} {n+1-k+j}.
\end{equation}
\end{subequations}

\begin{remark}\label{c=a-q=1:rem}
By adapting the proof of Theorem \ref{hwc-a:thm}, the cubature in Eqs. \eqref{mc:a}, \eqref{cw-mc:a}  can be readily inferred independently.
To this end it suffices to replace the orthogonality relations in Eqs. \eqref{or-c:a} and  \eqref{or-d:a}  by the corresponding $q=1$ degenerations:
\begin{equation*}
\sum_{\lambda\in\Lambda^{(m,n)}_{\texttt{a}}}      P_{\texttt{a};\mu} \left(\frac{2\pi\lambda}{m} ;1 \right)     \overline{  P_{\texttt{a};\nu} \left(\frac{2\pi \lambda}{m} ;1 \right) }
\delta^{(m,n)}_{\texttt{a};\lambda } (1)
=
\begin{cases} 
1/ \delta^{(m,n)}_{\texttt{a};\mu } (1)&\text{if}\  \nu= \mu  \\
0 &\text{if}\  \nu\neq \mu 
\end{cases} 
\end{equation*}
(cf. e.g. \cite[Section 5.2]{die-vin:quantum}), and
\begin{equation*}
\frac{1}{(2\pi )^{n-1} n^{1/2}}   \int_{\mathbb{A}^{(n)}_{\texttt{a}}}    P_{\texttt{a};\mu}  (\boldsymbol{\xi};1) \overline{P_{\texttt{a};\nu}  (\boldsymbol{\xi};1 )}     \text{d} 
\boldsymbol{\xi} 
=
\begin{cases} 
N_{\texttt{a};\mu} &\text{if}\  \nu= \mu \\
0 &\text{if}\  \nu\neq \mu 
\end{cases} 
\end{equation*}
($\mu,\nu\in\Lambda^{(m,n)}_{\texttt{a}}$).
\end{remark}

\subsection{Symmetrized quadrature}\label{SymmetrizedB:sec}
For $q=1$ the multivariate hyperoctahedral Hall-Littlewood polynomials factorize in terms of the corresponding univariate polynomials:
\begin{equation}
P_{\texttt{b};\mu}(\boldsymbol{\xi}; 1, q_0) = \prod_{1\leq j\leq n}   p_{\mu_j}(\xi_j;q_0) 
\end{equation}
(where $p_l(\xi;q_0)$ is taken from Eq. \eqref{HL-n=1:b}). Our cubature rule then becomes an $n$-fold product of quadratures restricted to the space of symmetric functions:
\begin{subequations}
\begin{align}\label{sq:b}
\frac{1}{(2\pi )^{n} }  & \int_{\mathbb{A}^{(n)}_{\texttt{b}}}   R_{\texttt{b}}(\xi_1,\ldots ,\xi_n)  \prod_{1\leq j\leq n} \rho_{\texttt{b}}(\xi_j )  \text{d} \xi_j 
= \\
&\sum_{\lambda\in\Lambda^{(m,n)}_{\texttt{b}}}      R_{\texttt{b}}\bigl( \xi_{\lambda_1}^{(m+1)},\ldots ,\xi_{\lambda_n}^{(m+1)} \bigr)   
\left( \prod_{1\leq j\leq n} \rho_{\texttt{b}}(\xi^{(m+1)}_{\lambda_j} ) \Delta^{(m+1)}_{\lambda_j } \right)     \delta_{\texttt{b};\lambda}^{(m,n)}(1) , \nonumber
\end{align}
where the integration measure  determined by $
 \rho_{\texttt{b}}(\xi )= 4(1-\cos^2 (\xi)) $,   the Christoffel weights  $\Delta^{(m+1)}_l$ and nodes $\xi^{(m+1)}_{\lambda_j} $ are governed by  Eqs. \eqref{cw-n=1:b} and \eqref{n=1-nodes:b}, and
 \begin{equation}
  \delta_{\texttt{b};\lambda}^{(m,n)}(1)   = \prod_{\substack{1\leq j<k\leq n\\ \lambda_j-\lambda_k=0}}  \frac{k-j}{1+k-j} .
  \end{equation}
The cubature rule in question is exact for symmetric rational functions with prescribes poles of the form
\begin{equation}\label{sq-R:b}
R_{\texttt{b}}(\xi_1,\ldots ,\xi_n)=   \frac{f(\xi_1,\ldots, \xi_n)}{\prod_{1\leq j\leq n}  (1-2q_0\cos(\xi_j) +q_0^2) }    ,
\end{equation}
\end{subequations}
where $f(\xi_1,\ldots ,\xi_n)$  denotes an arbitrary symmetric polynomial in $\mathbb{P}^{(2m,n)}_{\texttt{b}}$ (or in $\mathbb{P}^{(2m+1,n)}_{\texttt{b}}$ if $q_1=0$, cf. Remark \ref{jump:rem}).

\begin{remark}
As before (cf.  Remark \ref{c=a-q=1:rem}), the cubature in Eqs. \eqref{sq:b}---\eqref{sq-R:b}  is  readily verified
by adapting the proof of Theorem \ref{hwc-b:thm}. 
The relevant $q=1$ degenerations of the
orthogonality relations in Eqs.  \eqref{or-c:b} and \eqref{or-d:b} read:
\begin{align*}
\sum_{\lambda\in\Lambda^{(m,n)}_{\texttt{b}}}    
   \Biggl( P_{\texttt{b};\mu} \bigl( \xi^{(m+1)}_{\lambda_1},\ldots, \xi^{(m+1)}_{\lambda_n} ;1, q_0 \bigr)    
\overline{  P_{\texttt{b};\nu} \bigl(\xi^{(m+1)}_{\lambda_1},\ldots, \xi^{(m+1)}_{\lambda_n}  ;1, q_0 \bigr) }  \delta^{(m,n)}_{\texttt{b};\lambda }(1)  &\\
\times  \prod_{1\leq j\leq n} | c(\xi_{\lambda_j}^{(m+1)};q_0) |^{-2}
\Delta^{(m+1)}_{\lambda_j }  \Biggr)  
=
\begin{cases} 
1/ \delta^{(m,n)}_{\texttt{b};\mu }(1) &\text{if}\  \nu= \mu , \\
0 &\text{if}\  \nu\neq \mu ,
\end{cases} & \nonumber
\end{align*}
and
\begin{align*}
\frac{1}{(2\pi )^{n} }   \int_{\mathbb{A}^{(n)}_{\texttt{b}}}    P_{\texttt{b};\mu}  (\xi_1,\ldots ,\xi_n;1,q_0) \overline{P_{\texttt{b};\nu}  (\xi_1,\ldots ,\xi_n;1 ,q_0) }   \prod_{1\leq j \leq n} | c(\xi_j;q_0) |^{-2}  \text{d} 
\boldsymbol{\xi} &  \\
=
\begin{cases} 
1/ \delta^{(m,n)}_{\texttt{b};\mu }(1)  &\text{if}\  \nu= \mu ,\\
0 &\text{if}\  \nu\neq \mu .
\end{cases}
 & \nonumber
\end{align*}
\end{remark}

\begin{remark}
The cubature rule in Eqs. \eqref{sq:b}---\eqref{sq-R:b} amounts to a particular example of the symmetrized $n$-fold quadrature rule in
\cite{ber-sch-xu:multivariate} (cf. the formula in {\em loc. cit.} on the middle of page 31). The pertinent underlying orthogonal polynomials $p_l(\xi;q_0)$ \eqref{HL-n=1:b} can be identified
as one-parameter  Bernstein-Szeg\"o polynomials of the second kind \cite[Section 2.6]{sze:orthogonal}. For specific values of the parameter $q_0$, one reduces to  Chebyshev polynomials of the second kind ($q_0=0$), of the third kind ($q_0=-1$), or of the fourth kind ($q_0=1$), respectively (cf. e.g. \cite[Remark 6.1]{die-ems:exact}). When specializing $q_1$ in the same way, the associated quadratures stem
from the orthogonality relations of standard discrete (co)sine transforms:  DCT-2 ($q_0=1$, $q_1=1$), DCT-4 ($q_0=1$, $q_1=-1$), DCT-8 ($q_0=1$, $q_1=0$), DST-1 ($q_0=0$, $q_1=0$), DST-2 ($q_0=-1$, $q_1=-1$), DST-4 ($q_0=-1$, $q_1=1$), DST-5 ($q_0=0$, $q_1=-1$), DST-6 ($q_0=-1$, $q_1=0$), DST-7 ($q_0=0$, $q_1=1$), cf. e.g. \cite[Eq. (3.8)]{die-ems:cubature}.
A systematic study of some of these and other closely related symmetrized $n$-fold quadrature rules was carried out in Refs.
\cite{hri-mot:discrete,moo-mot-pat:gaussian,hri-mot-pat:cubature}.
\end{remark}

\bibliographystyle{amsplain}

\end{document}